\documentclass[11pt]{article}
\usepackage{geometry}\usepackage[latin1]{inputenc} 
\usepackage[english]{babel}
\usepackage{graphicx,xcolor}
\usepackage{amssymb,amsthm,mathtools}
\usepackage{hyperref,url}
\usepackage{tikz}
\usepackage[normalem]{ulem}

\usepackage[linesnumbered,ruled,vlined]{algorithm2e}
\SetKwInput{KwInput}{Input}
\SetKwInput{KwOutput}{Output}

\usepackage{cleveref}
\usepackage{comment} 
\usepackage{enumitem}

\newtheorem{proposition}{Proposition}
\newtheorem{lemma}{Lemma}
\newtheorem{theorem}{Theorem}

\newtheorem*{remark}{Remark}

\usepackage{ulem}

\usepackage{bm}                 
\DeclareMathOperator*{\argmin}{\textrm{argmin}\,} 
\def\Diag            {\textrm{Diag}} 	 
\def\prox            {\mathbf{P}}     
\def\rk              {\textrm{rank}}     
\def\st              {\,\textrm{s.t.}\,} 
\def\zeros           {\bm{0}}            
\def\IN {\mathbb{N}} 		
\def\IR  {\mathbb{R}} 	    
\def\IRm {\IR^{m}} 			
\def\IRn {\IR^{n}} 			
\def\IRmn{\IR^{m \times n}} 
\def\IRmr{\IR^{m \times r}} 
\def\IRrn{\IR^{r \times n}} 

\def\A {\bm{A}}

\def\D {\bm{D}}

\def\G {\bm{G}}

\def\L {\bm{L}}
\def\M {\bm{M}}
\def\N {\bm{N}}

\def\P {\bm{P}}

\def\S {\bm{S}}

\def\U {\bm{U}}
\def\V {\bm{V}}
\def\W {\bm{W}}
\def\X {\bm{X}} 
\def\Y {\bm{Y}}

\def\d {\bm{d}}

\def\s {\bm{s}}

\def\u {\bm{u}}
\def\v {\bm{v}}

\def\x {\bm{x}} 
\def\y {\bm{y}}
\def\z {\bm{z}}

\def\cP { \mathcal{P}}

\def\bSigma   {\boldsymbol{\Sigma}}

\def\bxi      {\boldsymbol{\xi}}

\usepackage{authblk}

\title{Binno: A 1st-order method for Bi-level Nonconvex Nonsmooth Optimization for Matrix Factorizations}
{
\author{\small 
Laura Selicato \vspace{-3ex}
\\ \small 
National Research Council (CNR), Water Research Institute (IRSA), Italy \vspace{-3ex}
\and \small 
Flavia Esposito \vspace{-3ex}
\\  \small 
University of Bari Aldo Moro, Department of Mathematics, Bari, Italy \vspace{-3ex}
\and \small 
Andersen Ang \vspace{-3ex}
\\ \small 
School of Electronics and Computer Science, University of Southampton, UK
} \vspace{-3em}
}

\begin{document}
\maketitle \vspace{-3ex}

\begin{abstract}
{Nonconvex and nonsmooth bi-level optimization poses critical theoretical challenges, while arising in several applications. In this work, we} develop a method for nonconvex, nonsmooth bi-level optimization and introduce \texttt{Binno}, a {first-order} method {that builds on proximal-gradient updates within the the proximal alternate minimization framework with descent conditions from variational analysis.}
{\texttt{Binno}} couples the two {levels via} a descent-driven averaging mechanism, extending {single-level} proximal schemes to the nonconvex {nonsmooth} bi-level setting.
We {show} that\texttt{Binno} induces a descent property for a suitable surrogate of the bi-level objective.
Each iteration performs blockwise proximal-gradient updates for the upper and lower problems separately{,} then forms a calibrated, block-diagonal convex combination of the two iterates.
A linesearch selects {combination} weights {enforcing} simultaneous descent of both {objectives. 
We} {give} conditions {ensuring that both these weights and the} descent directions induced by the associated proximal-gradient maps {exist}.
We apply \texttt{Binno} {to} sparse low-rank factorization, where the upper level uses elementwise $\ell_1$ penalties and the lower level uses nuclear norms, coupled via a Frobenius data term.
{We test \texttt{Binno} on synthetic matrices and a real traffic-video dataset, achieving lower reconstruction error and higher peak signal-to-noise ratio than standard methods. We also validate it on a regularized market-clearing problem, where it selects policy-preferred equilibria, and compare it with bi-level baseline, showing consistent improvements.}

\end{abstract}

\tableofcontents

\section{Introduction}\label{sec:intro}
Bi-level optimization problems consist of two nested { tasks} in a hierarchical structure.
The {upper level} determines variables that influence the {lower level}.
A solution {is optimal for both levels} under {this} hierarchy: the {upper-level} decision anticipates {the lower-level} optimal response.
{These problems  arise naturally} in decision science and learning{, where two optimization processes are coupled} \cite{dempe2020bilevel,del2023bi, del2025penalty}.
In this paper, we { focus on} the following bi-level problem:
\begin{equation}\label{eq:Binno}
\begin{array}{l}
\displaystyle
\argmin_{\x\in \IRn, \y \in \IRm}
\Big\{ 
\psi_1(\x,\y) \coloneqq  f_1(\x) + g_1(\y) + H(\x,\y)
\Big\}
\\
\displaystyle 
\quad \st~ (\x,\y) \in {\argmin_{\u\in \IRn, \v \in \IRm}
\Big\{ 
\psi_2(\u,\v) \coloneqq  f_2(\u) + g_2(\v) + H(\u,\v)
\Big\}},
\end{array}
\end{equation}
{where the upper and lower level problems, {denoted by} functions with subscript $i=1,2$, {minimize} nonconvex nonsmooth functions $\psi_i: \IRn\times\IRm\to(-\infty,+\infty]$
with
\begin{itemize}[leftmargin=15pt,itemsep=0pt]
\item $f_i:\IRn\to(-\infty,+\infty]$ and $g_i:\IRm\to(-\infty,+\infty]$ convex, proper and lower semicontinuous (l.s.c.);
\item $H:\IRn\times\IRm\to\IR$ {of class} $C^1$ block-wise ({i.e., in one variable at a time}).
\end{itemize}}
\noindent
{In} \eqref{eq:Binno}, the lower level problem has the same structure as the upper {level,} with $f_2 \neq f_1$ and $g_2 \neq g_1$.
Problem\eqref{eq:Binno} with $\psi_2 = 0$, {i.e., a single level,} frequently arises in machine {learning,} where {regularizers $f_1$ and $g_1$ encode constraints {on} variables $\x$ and $\y$.}
A classic example arises when these functions {are} indicator functions of half-space constraints \cite{combettes2011proximal,parikh2014proximal}.

{Classical methods for bi-level optimization include gradient-based implicit differentiation, hypergradient methods, and approximation schemes \cite{beck2023survey}. These approaches typically require smoothness of the lower-level problem or strong regularity assumptions.}
For convex inner problems and strongly convex outer objectives, first-order bi-level methods with provable { convergence} rates exist.
An {example} is the Sequential Averaging Method (SAM) and its bi-level specialization BiG-SAM, which {treat} bi-level optimization as a fixed-point selection problem and average a proximal-gradient step for the inner composite with a gradient step for the outer objective;
sublinear $O(\tfrac{1}{k})$ rates are known under standard Lipschitz/strong-convexity assumptions \cite{sabach2017first}.
Recent variants relax projections, incorporate inertial or conditional-gradient updates, or target ``simple" convex bi-level problems \cite{cao2024accelerated,giang2024projection}.
{There are also works that recently try to address bi-level formulations with nonsmooth inner problems\cite{alcantara2026theoretical,okuno2021lp}. 
However, these approaches typically focus on specific structures and do not address the general setting considered here, where both levels may involve nonsmooth and nonconvex components. Then, these techniques are not directly applicable to our setting, where both upper and lower levels involve nonsmooth composite structures and distinct regularizers, leading to potentially conflicting descent directions.} Two {gaps remain:} (i) most analyses assume convexity and smoothness at the outer level; (ii) existing SAM-type schemes average one upper-level step with one inner fixed-point map, while block-structured, composite, possibly nonconvex models with explicit proximal treatment at both levels have received less attention.

\subsection{Contribution and Paper Organization}
{We propose \texttt{Binno} (Bi-level nonconvex nonsmooth optimization), a bi-level generalization of Proximal Alternating
Linearized Minimization (PALM) to solve Problem \eqref{eq:Binno}.}
Motivated by PALM for nonconvex, nonsmooth single-level composites \cite{bolte2014proximal}, { \texttt{Binno} is a} bi-level generalization that {applies} blockwise proximal-gradient updates {at} both levels and then forms a calibrated convex combination of the upper- and lower-driven iterates to steer the sequence toward an upper-level preferred solution within the lower-level solution set.
This design preserves the modularity and per-block simplicity of proximal methods while embedding bi-level guidance directly into the iteration.

{The paper is organized as follow. Section~\ref{sec:background} provides notations and auxiliary results.}
{Section~\ref{sec:binno} details \texttt{Binno} and its theoretical properties.}
{Section~\ref{sec:SLR} applies \texttt{Binno} to sparse low-rank matrix factorization\footnote{{Throughout this paper, we use the term 'matrix factorization' to refer to the problem of approximating a given matrix through a low-rank factorized representation, which is the primary focus of our bi-level optimization method.}}.}
{Section~\ref{sec:exp} presents numerical experiments comparing \texttt{Binno} to standard methods on synthetic and real datasets.
Section~\ref{sec:RMC} presents numerical experiments comparing \texttt{Binno} to market clearing economics on real dataset.
}

\section{Background}\label{sec:background}
This work relies on the mathematical tools detailed below.

\paragraph{Proximal gradient update}
{For optimization problems of the form} \eqref{eq:Binno} with $\psi_2=0$, the proximal gradient method~\cite{martinet1970breve,combettes2011proximal,beck2017first} is {widely used}.
{It solves problems of the form}
\begin{equation}\label{opt:proxgrad}
\argmin_{\x \in \IRn} p(\x) + q(\x),
\end{equation}
with $p:\IRn \to \IR$ convex, $L$-smooth (gradient is $L$-Lipschitz) and $q :\IRn \to \IR \cup \{+\infty\}$ convex, proper and l.s.c..
At iteration $k \in \IN$, the proximal gradient update is $\x_{k+1} = \prox_{q}^\nu \big(  \x_k - \nu \nabla p(\x_k) \big)$
where $\nu >0$ is {the stepsize} and $\prox_{q}^\nu : \IRn \to \IRn$ is the proximal operator, defined as 
\[
\prox_{q}^\nu (\x) \coloneqq \argmin_{\bxi \in \IRn} 
\Big\{ q(\bxi) + \tfrac{1}{2\nu}\|\bxi-\x\|_2^2 \Big\},
\qquad
M_{\nu q} \coloneqq \min_{\bxi \in \IRn} 
\Big\{ q(\bxi) + \tfrac{1}{2\nu}\|\bxi-\x\|_2^2 \Big\},
\]
with $M_{\nu q}$ {the} Moreau envelope associated to $q$.
Under {suitable} assumptions on $p,q,\nu$ (see \cite{beck2017first}), the sequence $\{\x_k\}_{k \in \IN}$ {produced} converges to a global minimizer of {Problem}~\eqref{opt:proxgrad}.

The following lemma {will be useful}.
\begin{lemma}[Theorem~12.30, \cite{Bauschke2017}]\label{lem:M_grad_boounded}
If $f$ is convex, proper and l.s.c., then $M_{\nu f}\in C^1$ and for all $\x\in\IRn$, $\nabla M_{\nu f}=\frac{1}{\nu}(\x-\prox_{\nu f}(\x))$.
Consequently, its gradient is $\tfrac{1}{\nu}$ Lipschitz continuous.
\end{lemma}

{We denote by $\G$ the proximal gradient map associated with~\eqref{opt:proxgrad}:} {letting} $\partial q$ {denote} a subgradient of {$q$},
\[
\begin{array}{rclcl}
\x_+ 
&=& 
\prox_q^\nu( \x - \nu \nabla p(\x))
&=& 
\x - \nu \G(\x)
\\[1mm]
\G(\x)
&=&
\frac{1}{\nu}\big(
    \x - \prox_q^\nu( \x - \nu \nabla p(\x))
    \big)
&\in& 
\nabla p(\x) + \partial q(\x - \nu \G(\x)  ).
\end{array}
\]

\paragraph{Proximal Alternating Linearized Minimization (PALM)}
{Given} the structure of each level of {Problem}~\eqref{eq:Binno}, we recall {the} PALM algorithm~\cite{bolte2014proximal}{, which extends the proximal gradient update to the two-block problem}
\[
\argmin_{\x \in \IRn, \y \in \IRm} f(\x) + g(\y) + H(\x,\y),
\]
{obtained by dropping the bi-level structure in~\eqref{eq:Binno}; $f,g$ satisfy the same assumptions as $q$ in~\eqref{opt:proxgrad} and $H\in C^1$.}
{PALM applies the proximal gradient update alternately on each subproblem of} \eqref{eq:Binno}:
\[
\x_{k+1} = \prox_{f}^\nu \big(  \x_k - \nu \nabla_{\x} H(\x_k,\y_k) \big),
\qquad
\y_{k+1} = \prox_{g}^\nu \big(  \y_k - \nu \nabla_{\y} H(\x_{k+1},\y_k) \big).
\]
{The sequence generated by PALM {converges} to a critical point of the objective {provided, it satisfies} the Kurdyka-Lojasiewicz (KL) property and the partial gradients $\nabla_{\x}H$ and $\nabla_{\y}H$ are globally Lipschitz.}

\paragraph{Sequential Averaging Method (SAM)}
{SAM targets Problem~\eqref{opt:proxgrad} in a bi-level setting: among its minimizers it selects the one solving an outer objective. Writing the inner problem with $f,g$ in place of $p,q$,}
\begin{equation}\label{prob:SAM}
\min_{\x} \big\{ \varphi(\x)=f(\x)+g(\x) \big\}, 
\end{equation}
{the outer problem is $\min_{\x\in X^\star}\omega(\x)$ with $\omega$ strongly convex, smooth, and $\nabla\omega$ $L_\omega$-Lipschitz, where $X^\star$ is the solution set of~\eqref{prob:SAM}.
Given} a nonexpansive map $T:\IRn\to\IRn$ and a contraction $S:\IRn\to\IRn$, SAM iterates
\[\textstyle
\x_k = \alpha_kS(\x_{k-1}) + (1-\alpha_k)T(\x_{k-1}),\quad \text{with} 
~
\alpha_k\in(0,1],
~
\alpha_k\downarrow0,
~
\sum_k\alpha_k=\infty,
\]
and converges to a point $\x^\star\in\operatorname{Fix}(T)
\coloneqq  \{ \x \in \IRn :  T(\x) = \x\}$ {satisfying} the variational inequality
$\langle \x^\star - S(\x^\star), \x - \x^\star\rangle \ge 0$ for all $\x\in\operatorname{Fix}(T)$.
{A choice is to take $T$ as the proximal gradient operator of the inner problem and $S$ as the gradient-step contraction for the outer problem: $T(\x)= \x - t\G(\x) = \prox^t_{g}\bigl(\x - t\nabla f(\x)\bigr)$ with $t\in(0,\tfrac{1}{L_f}]$.
The operator $T$ is nonexpansive with $\operatorname{Fix}(T)=X^\star$; while $S(\x)=\x - s\nabla \omega(\x)$ with $s\in\big(0,\tfrac{2}{L_\omega+\sigma}\big]$. With these, SAM reads}
{
\[
\z_k = \x_{k-1}-s\nabla \omega(\x_{k-1}),
\qquad
\x_k = \alpha_k \z_k + (1-\alpha_k)
\prox^t_g\bigl(\x_{k-1}-t\nabla f(\x_{k-1})\bigr),
\]
}
and the iterates converge to $\x^\star\in X^\star$ solving the bi-level task via the first-order optimality condition
$\langle \nabla\omega(\x^\star), \x - \x^\star\rangle \ge 0$ for all $\x\in X^\star$.

{The above concepts are the building blocks for \texttt{Binno}:  PALM-type updates are used within the averaging mechanism of SAM.}

\section{Binno}\label{sec:binno}
The idea of \texttt{Binno} is to {alternately update each variable block by approximately solving each single-level subproblem via a proximal gradient update, and then taking a convex combination of the two resulting sequences}. {\texttt{Binno} can be seen as a PALM-based update embedded in a SAM-type averaging scheme, where two blockwise proximal-gradient dynamics are coupled through a convex combination designed to enforce bi-level descent.
Unlike SAM, which relies on a fixed-point averaging between a contraction and a nonexpansive map {for} a single objective, \texttt{Binno} operates on two coupled nonconvex objectives and uses a data-dependent convex combination of two PALM-generated dynamics to enforce simultaneous descent at both levels.}
{Specifically}, starting from an initial guess $(\x_0,\y_0)$, at each iteration $k$
we perform one PALM step on the upper level problem, with {subscript} $u$, {obtaining} $(\x_u, \y_u)$.
Similarly, we perform one PALM step on the lower level problem ({ with subscript} $\ell$), obtaining $(\x_{\ell}, \y_{\ell})$.
{We then obtain} $(\x_{k+1},\y_{k+1})$ {via} a convex combination of $(\x_u,\y_u)$ and $(\x_{\ell},\y_{\ell})$.
Fig.\ref{fig:flow_chart} {shows} a flow chart of the {iteration}, highlighting some issues that emerge ({indicated as $Q_1$ and $Q_2$}).
We let $\tilde{\x}$ be the gradient-only update of $\x$ (i.e., before applying the prox operator).
Performing a simply convex combination, like
in SAM, is not appropriate in this setting.
{The following issues arise}:
\begin{enumerate}[leftmargin=15pt,itemsep=0pt]
\item {Since} $f_1\neq f_2$, the upper- and lower-level {updates} of $\x$ target different proximal maps, so $\x_u$ and $\x_{\ell}$ {generally point to} distinct fixed points.
This {complicates both the analysis and the effect of averaging them}.

\item {Since} $g_1\neq g_2$, the updates of $\y$ are {computed} from different $\x$ iterates ($\y_u$ uses $\x_u$ while $\y_{\ell}$ uses $\x_{\ell}$).
This cross-level coupling {yields} potentially conflicting descent directions and {intricate dynamics} for $\y$.
\end{enumerate}

\begin{figure}[h!]
\centering 
\begin{tikzpicture}{scale=1}
    \node (y_k) at (-2,0) {$\y_k$};
    \node (x_k) at (-2,1) {$\x_k$};
    \node (x_tile) at (0,1) {$\tilde{\x}$};
    \node (x_u) at (2,1) {$\x_u$};
    \node (y_tile) at (2,0) {$\tilde{\y}$};
    \node (y_u) at (8,0) {$\y_u$};
    
    \draw[->,very thick] (x_k) -- (x_tile) node[midway, above] {$\nabla_{\x}H$};
    \draw[->,very thick] (x_tile) -- (x_u) node[midway, above] {$\prox_{f_1}$};
    \draw[->,very thick] (y_k) -- (x_tile);
    \draw[->,very thick] (y_k) -- (y_tile);
    \draw[->,very thick] (x_u) -- (y_tile) node[midway, right] {$\nabla_{\y}H$};
    \draw[->,very thick] (y_tile) -- (y_u) node[midway, above, yshift =-1mm] {$\prox_{g_1}$};

    \node (y_kl) at (-2,-2) {$\y_k$};
    \node (x_kl) at (-2,-1) {$\x_k$};
    \node (x_tilel) at (0,-1) {$\tilde{\x}$};
    \node (x_l) at (2,-1) {$\x_{\ell}$};
    \node (y_tilel) at (2,-2) {$\tilde{\y}$};
    \node (y_l) at (8,-2) {$\y_{\ell}$};

    \draw[->,very thick] (x_kl) -- (x_tilel) node[midway, above] {$\nabla_{\x}H$};
    \draw[->,very thick] (x_tilel) -- (x_l) node[midway, above] {$\prox_{f_2}$};
    \draw[->,very thick] (y_kl) -- (x_tilel);
    \draw[->,very thick] (y_kl) -- (y_tilel);
    \draw[->,very thick] (x_l) -- (y_tilel) node[midway, right] {$\nabla_{\y}H$};
    \draw[->,very thick] (y_tilel) -- (y_l) node[midway, above, yshift =-1mm] {$\prox_{g_2}$};

    \node (y_k1) at (10,-2) {$\y_{k+1}$};
    \node (x_k1) at (10,1) {$\x_{k+1}$};
    \draw[->,very thick, dashed] (x_l) -- (x_k1) node[pos=0.75, above] {$Q_1$};
    \draw[->,very thick, dashed] (x_u) -- (x_k1) node[pos=0.75, above] {$Q_1$};
    \draw[->,very thick, dashed] (y_l) -- (y_k1) node[midway, above] {$Q_2$};
    \draw[->,very thick, dashed] (y_u) -- (y_k1) node[midway, above] {$Q_2$};
\end{tikzpicture}
\caption{The structure of \texttt{Binno}.
$Q_1$ and $Q_2$ {denote the issues of the problem solved with} the \texttt{Binno}-specific update, see \eqref{eq:convexcomb}.
}
\label{fig:flow_chart}
\end{figure}

\subsection{Theory of \texttt{Binno}}\label{proposed_approach}

At each iteration $k$, we {compute two PALM-type updates, one for the upper level and one for the lower level, and then {combine} them through a convex averaging step in the spirit of SAM.
This hybrid PALM-SAM mechanism enforces simultaneous descent of both objectives and {is the key distinction from} a simple repeated application of PALM.
{The steps are as follows}.}
\begin{enumerate}[leftmargin=15pt,itemsep=0pt]
    \item At the upper level subproblem, we perform a proximal gradient descent (ProxGrad) step on $\x_k$ {with} $\y_k$ {held fixed}
    \begin{equation}\label{eq:upperx}
    \x_u = \prox_{f_1}^\nu \big(
    \x_k - \nu \nabla_{\x}H(\x_k,\y_k)
    \big),
    \end{equation}
    where $\nu > 0$ is a stepsize and $\prox_{f_1}^\nu$ is the prox operator of {$f_1$} under parameter $\nu$.
    {Since this step is at the upper level, we label it $\x_u$.}

    \item At the upper level subproblem, we perform a ProxGrad step on $\y_k$ {with} $\x$  held fixed at the most recent value $\x_u$
    \begin{equation}\label{eq:uppery}
    \y_u = \prox_{g_1}^\nu \big(
    \y_k - \nu \nabla_{\y}H(\x_u,\y_k)
    \big).
    \end{equation}

    \item At the lower level subproblem, we perform a ProxGrad step on $\x_k$ {with} $\y_k$ held {fixed}
    \begin{equation}\label{eq:lowerx}
    \x_{\ell} = \prox_{f_2}^\nu \big(
    \x_k - \nu \nabla_{\x}H(\x_k,\y_k)
    \big).
    \end{equation}    

    \item At the lower level subproblem, we perform a ProxGrad step on $\y_k$ {with} $\x$ held fixed at the most recent value $\x_{\ell}$
    \begin{equation}\label{eq:lowery}
    \y_{\ell} = \prox_{g_2}^\nu \big(
    \y_k - \nu \nabla_{\y}H(\x_{\ell},\y_k)
    \big).
    \end{equation}

    \item We obtain {$(\x_{k+1},\y_{k+1})$ by a} convex combination of $(\x_u,\y_u)$ and $(\x_{\ell},\y_{\ell})$:
   \begin{equation}\label{eq:convexcomb}
    \begin{pmatrix}
    \x_{k+1} 
    \\
    \y_{k+1}
    \end{pmatrix}
    =
    \begin{pmatrix}
    \alpha_k I_n & \zeros 
    \\
    \zeros & \beta_k I_m
    \end{pmatrix}
    \begin{pmatrix}
    \x_u
    \\
    \y_u
    \end{pmatrix}
    +
    \begin{pmatrix}
    (1-\alpha_k) I_n & \zeros 
    \\
    \zeros & (1-\beta_k) I_m
    \end{pmatrix}
    \begin{pmatrix}
    \x_{\ell}
    \\
    \y_{\ell}
    \end{pmatrix},
    \end{equation}
     {where $I_n$, $I_m$ are identity matrices,}
    and $\alpha_k, \beta_k$ are {constants to be chosen},

\end{enumerate}

Algorithm~\ref{alg:binno} summarizes {these steps in pseudo-code}.

\begin{algorithm}[h!]\label{alg:binno}
\DontPrintSemicolon
\For{$k=0,1,...,$}  
{ 
\textbf{Upper-level update:} $(\x_u,\y_u) \overset{\eqref{eq:upperx}, \eqref{eq:uppery}}{\leftarrow}  (\x_k,\y_k)$ 

\textbf{Lower-level update:} $(\x_\ell,\y_\ell) \overset{\eqref{eq:lowerx}, \eqref{eq:lowery}}{\leftarrow} (\x_k,\y_k)$.

\textbf{Convex combination:} $(\x_{k+1},\y_{k+1}) \overset{\eqref{eq:convexcomb}}{\leftarrow} (\x_u,\y_u),(\x_\ell,\y_\ell)$.
}
\caption{\texttt{Binno} for Problem\eqref{eq:Binno} 
with initialization $\x_0, \y_0$}
\end{algorithm}
We use a line search scheme {to obtain} $\alpha_k, \beta_k$ such that {both} $\psi_1$ and $\psi_2$ {simultaneously descend}:
\[
\psi_1(\x_{k+1},\y_{k+1}) \leq \psi_1(\x_k,\y_k)
~~\text{ and }~~
\psi_2(\x_{k+1},\y_{k+1}) \leq \psi_2(\x_k,\y_k).
\]
\begin{remark}\label{remark:descent}
{The goal of the construction is to identify directions $\d_{\x}$ and $\d_{\y}$ that yield simultaneous descent for both the upper- and lower-level objectives. Since the proximal-gradient mappings associated with each level generally induce different descent directions, we construct $\d_{\x}$ and $\d_{\y}$ as convex combinations of such directions and enforce descent through suitable conditions on the combination parameters.}
Using the proximal gradient map (see \cref{sec:background}),
    \[
    \begin{array}{rcl}
    \x_{k+1}
    = \alpha \x_u + (1-\alpha) \x_{\ell}
    &=&
    \alpha  \prox_{f_1}^\nu \big(
    \x_k - \nu \nabla_{\x}H(\x_k,\y_k)
    \big) 
    + (1-\alpha) \prox_{f_2}^\nu \big(
    \x_k - \nu \nabla_{\x}H(\x_k,\y_k)
    \big)
    \\[1mm]
    &=& \alpha \big( \x_k - \nu \G^{u}_f(\x_k) \big) + (1-\alpha) \big( \x_k - \nu \G^{\ell}_f(\x_k) \big)
    \\[1mm]
    &=& \x_k - \nu \big( \alpha \G^{u}_f(\x_k)  + (1-\alpha) \G^{\ell}_f(\x_k) \big)
    \\[1mm]
    &=& \x_k + \nu \d_{\x} ~~\text{ with }~~ \d_{\x}= -\alpha \G^{u}_f(\x_k)  - (1-\alpha) \G^{\ell}_f(\x_k).
    \\[1mm]
    \y_{k+1}
    = \beta_k \y_u + (1-\beta_k) \y_{\ell}
    &=&
    \beta_k  \prox_{g_1}^\nu \big(
    \y_k - \nu \nabla_{\y}H(\x_u,\y_k)
    \big) 
    + (1-\beta_k) \prox_{g_2}^\nu \big(
    \y_k - \nu \nabla_{\y}H(\x_{\ell},\y_k)
    \big)
    \\[1mm]
    &=& \beta_k \big( \y_k - \nu \G^u_g(\y_k) \big) + (1-\beta_k) \big( \y_k - \nu \G^{\ell}_{g}(\y_k) \big)
    \\[1mm]
    &=& \y_k - \nu \big( \beta_k \G^u_g(\y_k)  + (1-\beta_k) \G^{\ell}_{g}(\y_k) \big)
    \\[1mm]
    &=& \y_k + \nu \d_{\y} ~~\text{ with }~~ \d_{\y}= -\beta_k \G^u_g(\y_k)  - (1-\beta_k) \G^{\ell}_{g}(\y_k).
    \end{array}
    \]
    where 
   {
  \[
   \begin{array}{ll}
   \G^{u}_f(\x_k) 
   = 
   \frac{1}{\nu}
   \big(\x_k-
   \prox_{f_1}^\nu \big(
    \x_k - \nu \nabla_{\x}H(\x_k,\y_k)
    \big)\big),
    & \qquad
   \G^{\ell}_f(\x_k)  
   =
    \frac{1}{\nu}
       \big(\x_k-
       \prox_{f_2}^\nu \big(
        \x_k - \nu \nabla_{\x}H(\x_k,\y_k)
        \big)\big),
    \\[1mm]
   \G^u_g(\y_k)  
   =
   \frac{1}{\nu}
   \big(\y_k-
   \prox_{g_1}^\nu \big(
    \y_k - \nu \nabla_{\y}H(\x_u,\y_k)\big)
    \big),
    & \qquad
   \G^{\ell}_{g}(\y_k)  
   =
   \frac{1}{\nu}
   \big(\y_k-
   \prox_{g_2}^\nu \big(
    \y_k - \nu \nabla_{\y}H(\x_{\ell},\y_k)\big)
    \big).
    \end{array}
    \]}
    {The proximal-gradient updates {for} the upper and lower problems generally induce different descent directions, {which need not guarantee} simultaneous decrease of both objectives. {We therefore impose that}} $\d_{\x}$ and $\d_{\y}$ are descent directions:
    \[
    \begin{array}{rll}
    \text{for upper level wrt $\x$ if} & \big\langle \partial \psi_1 (\x_k,\y_k), ~ \d_{\x} \big\rangle {<} 0
     &\iff~~
    \big\langle \partial f_1(\x_k) + \nabla_{\x}H(\x_k,\y_k), ~ \d_{\x} \big\rangle {<} 0,
    \\[1mm]
    \text{for upper level wrt $\y$ if} & \big\langle \partial \psi_1 (\x_u,\y_k), ~ \d_{\y} \big\rangle {<} 0
    &\iff~~
    \big\langle \partial g_1(\y_k) + \nabla_{\y}H(\x_u,\y_k), ~ \d_{\y} \big\rangle {<} 0,
    \\[1mm]
    \text{for lower level wrt $\x$ if} & \big\langle \partial \psi_2 (\x_k,\y_k), ~ \d_{\x} \big\rangle {<} 0
    &\iff~~
    \big\langle \partial f_2(\x_k) + \nabla_{\x}H(\x_k,\y_k), ~ \d_{\x} \big\rangle {<} 0,
    \\[1mm]
    \text{for lower level wrt $\y$ if} & \big\langle \partial \psi_2 (\x_{\ell},\y_k), ~ \d_{\y} \big\rangle {<} 0
    &\iff~~
    \big\langle \partial g_2(\y_k) + \nabla_{\y}H(\x_{\ell},\y_k), ~ \d_{\y} \big\rangle {<} 0.
 \end{array}
    \]
{Thus, to resolve $Q_1$, $Q_2$ in Fig.~\ref{fig:flow_chart}, we characterize conditions under which the combined directions $\d_{\x}$ and $\d_{\y}$ are simultaneous descent directions for both $\psi_1$ and $\psi_2$. This leads to the following theorem, which establishes the existence of $(\alpha_k, \beta_k)$ enforcing such a property.}
\end{remark}

{\begin{theorem}\label{th:descend_condition}
In the setting of Problem~\eqref{eq:Binno}, let all {the above} assumptions hold. Suppose {the} parameters $\alpha_k\in [0,1], \beta_k\in[0,1]$ satisfy conditions \eqref{conditions_alpha1}-\eqref{conditions_beta2}.
Then the {iterates produced by} \texttt{Binno} satisfy the descent properties
    \[
   \begin{array}{rcl}
    \big\langle \partial f_1(\x_k) + \nabla_{\x}H(\x_k,\y_k), ~~ \alpha_k \G^{u}_f(\x_k)  + (1-\alpha_k) \G^{\ell}_f(\x_k) \big\rangle > 0 
    \\[1mm]
    \big\langle \partial f_2(\x_k) + \nabla_{\x}H(\x_k,\y_k), ~~ \alpha_k \G^{u}_f(\x_k)  + (1-\alpha_k) \G^{\ell}_f(\x_k) \big\rangle > 0,
    \\[1mm]
        \big\langle \partial g_1(\y_k) + \nabla_{\y}H(\x_u,\y_k), ~~ \beta_k \G^u_g(\y_k)  + (1-\beta_k) \G^{\ell}_{g}(\y_k) \big\rangle > 0,
    \\[1mm]
    \big\langle \partial g_2(\y_k) + \nabla_{\y}H(\x_{\ell},\y_k), ~~ \beta_k \G^u_g(\y_k)  + (1-\beta_k) \G^{\ell}_{g}(\y_k) \big\rangle > 0.
   \end{array}
    \]
\end{theorem}}
{
\begin{remark}
Conditions \eqref{conditions_alpha1}-\eqref{conditions_beta2} define a feasible region for {$(\alpha_k,\beta_k)$}. {Although} the theorem is stated in conditional form, {this region is nonempty in practice under standard assumptions}, and the linesearch procedure is designed to identify admissible parameters at each iteration. For practical cases, see Section \ref{sec:SLR}.
\end{remark}}

To prove Theorem \ref{th:descend_condition}, we need some preliminary results.

\begin{lemma}\label{lem:nesterov}
Let $z$ be {a} convex, proper, l.s.c. function, and $\x_0 \in \text{int }(\text{dom }z)$.
Then $\partial z (\x_0)$ is a nonempty bounded set, i.e. there exists a constant $c$ such that $\|\partial z(\x_0)\|\leq c$.
\begin{proof}
{For} a closed convex function, $\partial z (x_0)$ is a nonempty bounded set \cite[Theorem 3.1.15]{Nesterov2018}.
A proper convex function is closed {iff} l.s.c., {the result follows}.
\end{proof}
\end{lemma}

\begin{lemma}\label{lem:subdiffbound}
Under the assumptions and settings of Theorem~\ref{th:descend_condition},
\[
\begin{array}{ll}
\big| \big\langle \partial f_1(\x_k),  \G_{\Delta} (\x_k) \big\rangle \big|
\leq
c_1\|\G_{\Delta}(\x_k)\|, 
&
\big| \big\langle \partial f_2(\x_k),  \G_{\Delta} (\x_k) \big\rangle \big|
\leq
c_2\|\G_{\Delta}(\x_k)\|, 
\\[1mm]
\big| \big\langle \partial g_1(\y_k),  \G_{\Delta} (\y_k) \big\rangle \big|
\leq
c_3\|\G_{\Delta}(\y_k)\|, 
&
\big| \big\langle \partial g_2(\y_k),  \G_{\Delta} (\y_k) \big\rangle \big|
\leq
c_4\|\G_{\Delta}(\y_k)\|,
\end{array}
\]
{where $\G_{\Delta}\in \{\G^{u}_f,\G^u_g,\G^{\ell}_f,\G^{\ell}_{g}\}$} and  $c_1,c_2,c_3,c_4$ are constants for functions $f_1,f_2,g_1,g_2$ respectively, as in Lemma \ref{lem:nesterov}.
\begin{proof}
We prove the lemma for $f_1$.
By {Cauchy-Schwarz}:
\[
\big|\big\langle \partial f_1(\x_k), \G_{\Delta} (\x_k) \big\rangle \big| 
\leq
\|\partial f_1(\x_k)\|  \|\G_{\Delta} (\x_k)\| 
~\overset{\cref{lem:nesterov}}{\leq}~
c_1 \|\G_{\Delta} (\x_k)\|.
\]  
The {remaining cases} are similar {with} constants $c_2,c_3,c_4$.
\end{proof}
\end{lemma}

\begin{lemma}\label{lem:H_bismooth}
Under the assumptions and settings of Theorem~\ref{th:descend_condition}, {let} $L_1,L_2$ {be the} bi-smooth constants for the gradient of $H$ {in} $\x$ and $\y$, respectively. Then
\[
\begin{array}{rcl}
\big|\big\langle \nabla_{\x} H(\x_k,\y_k),  \G_{\Delta} (\x_k) \big\rangle \big|
&<&
L_1\|\G_{\Delta}(\x_k)\|,
\\[1mm]
\big| \big\langle \nabla_{\y} H(\x_{\Delta},\y_k),  \G_{\Delta} (\y_k) \big\rangle \big|
&<&
L_2\|\G_{\Delta}(\y_k)\|, \quad \x_{\Delta}\in\{\x_u,\x_{\ell}\}.
\end{array}
\]
\begin{proof}
We prove the {case} for $\x$.
{Since} $H$ is bi-differentiable and bi-smooth, $ \|\nabla_{\x}H(\x_k,\y_k)\| \leq L_1$.
By the Cauchy-Schwarz inequality,
\[
\Big|\big\langle \nabla_{\x}H(\x_k,\y_k), \G_{\Delta} (\x_k) \big\rangle\Big| 
\leq
\|\nabla_{\x}H(\x_k,\y_k)\|  \| \G_{\Delta} (\x_k)\| 
\leq
L_1 \| \G_{\Delta} (\x_k)\|.
\] 
The proof for $\y$ is similar with $\nabla_{\y}H(\x_{\Delta},\y_k)$ and constant $L_2$.
\end{proof}
\end{lemma}

\begin{lemma}\label{lem:g}
Under the assumptions and settings of Theorem~\ref{th:descend_condition}, $\|\G_{\Delta}(\x_k)\|$ and $\|\G_{\Delta}(\y_k)\|$ are bounded, {where $\G_{\Delta}\in \{\G^{u}_f,\G^u_g,\G^{\ell}_f,\G^{\ell}_{g}\}$}.
\begin{proof}
We prove the {case} for $\G^{u}_f(\x_k)$; the rest is similar.
The prox operator is a contraction,
\[
\|\prox_{f}(\s) - \prox_{f}(\z)\| \leq \|\s - \z\|.
\]
{Taking} $f=f_1${,} $\s = \x_k$ and $\z= \x_k - \nu \nabla_{\x}H(\x_k,\y_k))${,}
\begin{equation}\label{lem:g_1}
\big\|\prox^{\nu}_{f_1}(\x_k)- \prox^{\nu}_{f_1}\big( \x_k - \nu \nabla_{\x}H(\x_k,\y_k) \big) \big\|
\leq
\| \nu \nabla_{\x}H(\x_k,\y_k)\| 
\overset{H \text{ smooth}}{\leq}
\nu L_1.
\end{equation}
Then 
\[
\begin{array}{rl}
\|\G^{u}_f(\x_k)\|
=&\hspace{-3mm}
\Big\|
\tfrac{1}{\nu}
\Big(
\x_k-\prox^{\nu}_{f_1}(\x_k)+\prox^{\nu}_{f_1}(\x_k)-\prox^{\nu}_{f_1}\big(x_k-\nu\nabla_xH(\x_k,\y_k)\big)
\Big)
\Big\|
\\[3mm]
\leq&\hspace{-3mm}
\tfrac{1}{\nu}
\Big[
\|
\x_k-\prox^{\nu}_{f_1}(\x_k)\|
+ \big\| 
\prox^{\nu}_{f_1}(\x_k)
-\prox^{\nu}_{f_1} \big(x_k-\nu\nabla_xH(\x_k,\y_k)\big)
\big\|
\Big]
\\[1mm]
\overset{\eqref{lem:g_1}}{\leq}&\hspace{-3mm}
\tfrac{1}{\nu}\big[
\|\x_k-\prox^{\nu}_{f_1}(\x_k)\|  +\nu L_1
\big]
=\hspace{-3mm}
\underbrace{\tfrac{1}{\nu}\|\x_k-\prox^{\nu}_{f_1}(\x_k)\|
}_{=\|\nabla M_{\nu f_1}(\x_k)\|}
+
L_1
\overset{\cref{lem:M_grad_boounded}}{\leq}~ \tfrac{1}{\nu}+L_1.
\end{array}
\]
\end{proof}
\end{lemma}

%
%
We are now ready to prove Theorem~\ref{th:descend_condition}.
\begin{proof}[Proof of Theorem~\ref{th:descend_condition}]
We focus on $\x_k$; {the proof for} $\y_k$ {is similar}.
First,
\[
\begin{array}{l}
\big\langle \partial f_1(\x_k) + \nabla_{\x}H(\x_k,\y_k), ~ \alpha_k \G^{u}_f(\x_k)  + (1-\alpha_k) \G^{\ell}_f(\x_k) \big\rangle 
   \\
   = \alpha_k 
   \underbrace{
   \big\langle \partial f_1(\x_k) + \nabla_{\x}H(\x_k,\y_k), ~ \G^{u}_f(\x_k)\big\rangle
   }_{\coloneqq q_1}
+ (1-\alpha_k) \big\langle \partial f_1(\x_k) + \nabla_{\x}H(\x_k,\y_k), ~\G^{\ell}_f(\x_k) \big\rangle,
   \end{array}
\]
\[
\begin{array}{l}
   \big\langle \partial f_2(\x_k) + \nabla_{\x}H(\x_k,\y_k),  ~ \alpha_k \G^{u}_f(\x_k)  + (1-\alpha_k) \G^{\ell}_f(\x_k) \big\rangle
   \\
   = \alpha_k 
   \big\langle \partial f_2(\x_k) + \nabla_{\x}H(\x_k,\y_k), ~ \G^{u}_f(\x_k) \big\rangle
 + (1-\alpha_k) \underbrace{\big\langle \partial f_2(\x_k) + \nabla_{\x}H(\x_k,\y_k),  ~\G^{\ell}_f(\x_k) \big\rangle}_{\coloneqq q_2}.
\end{array}
\]
$q_1 >0$ {because} $-\G^{u}_f(\x_k)$ is a descent direction for {the} upper problem $\psi_1$, disregarding the lower problem.
{Similarly} $q_2 > 0$.

For the other parts,
\[
\begin{array}{cl}
&\hspace{-2cm} \big\langle \partial f_1(\x_k), \G^{\ell}_f(\x_k) \big\rangle  + \big\langle \nabla_{\x}H(\x_k,\y_k), \G^{\ell}_f(\x_k) \big\rangle
\\[1mm]
{\geq} &
{- \Big(} \big|\big\langle \partial f_1(\x_k), \G^{\ell}_f(\x_k) \big\rangle\big|
+ 
\big|\big\langle \nabla_{\x}H(\x_k,\y_k), \G^{\ell}_f(\x_k) \big\rangle\big| {\Big)}
\\[1mm]
\overset{\cref{lem:subdiffbound}, \ref{lem:H_bismooth}}{\geq}&
{- \Big(} c_1\|\G^{\ell}_f(\x_k)\|+L_1\|\G^{\ell}_f(\x_k)\| {\Big)}
=
{-} (c_1 + L_1 )\|\G^{\ell}_f(\x_k)\|
~\overset{\cref{lem:g}}{\geq}~
{-} (c_1 + L_1 )\Big(\tfrac{1}{\nu}+L_1\Big).
\end{array}
\] 
{Hence} $k_1\coloneqq (c_1 + L_1 )\big(\frac{1}{\nu}+L_1\big)$ {exists such that}
\[
\big\langle \partial f_1(\x_k)+ \nabla_{\x}H(\x_k,\y_k), ~ \G^{\ell}_f(\x_k)\big\rangle 
{\geq -k_1}.
\] 
Similarly, {there exists} $k_2 \coloneqq (c_2+L_1)\big(\frac{1}{\nu}+L_1\big)$ such that 
\[
\big\langle \partial f_2(\x_k)+ \nabla_{\x}H(\x_k,\y_k), ~ \G^{u}_f(\x_k)\big\rangle 
{\geq -k_2}.
\]
Finally, {with} $\nabla_{\x}H_k = \nabla_{\x}H(\x_k,\y_k)$,
\[
\begin{array}{l}
\alpha_k \big\langle \partial f_1(\x_k) + \nabla_{\x}H_k, ~  \G^{u}_f(\x_k)\big\rangle 
+ (1-\alpha_k) \big\langle \partial f_1(\x_k) + \nabla_{\x}H_k,  ~ \G^{\ell}_f(\x_k) \big\rangle  
{\geq} 
\alpha_k q_1  -k_1 (1-\alpha_k),
\\[2mm]
\alpha_k \big\langle \partial f_2(\x_k) + \nabla_{\x}H_k, ~ \G^{u}_f(\x_k) \big\rangle 
+ (1-\alpha_k) \big\langle \partial f_2(\x_k) + \nabla_{\x}H_k, ~ \G^{\ell}_f(\x_k) \big\rangle
{\geq}  {-k_2} \alpha_k  + (1-\alpha_k) q_2.
\end{array}
\]
Similarly for $\y$ and $\beta_k$:
\begin{itemize}[leftmargin=15pt,itemsep=0pt]
\item for the upper problem: {setting} $q_3 \coloneqq \big\langle \partial g_1(\y_k) + \nabla_{\y}H(\x_u,\y_k), \G^u_g(\y_k)\big\rangle$ and $k_3 \coloneqq (c_3+L_2)\big(\frac{1}{\nu}+L_2\big)$,
\[
   \begin{array}{lll}
   \hspace{-1.5cm}
   \big\langle \partial g_1(\y_k) + \nabla_{\y}H(\x_u,\y_k),  ~~ \beta_k \G^u_g(\y_k)  + (1-\beta_k) \G^{\ell}_{g}(\y_k) \big\rangle 
\\[1mm]
   = \beta_k 
   \big\langle \partial g_1(\y_k) + \nabla_{\y}H(\x_u,\y_k), \G^u_g(\y_k)\big\rangle
    + (1-\beta_k) \big\langle \partial g_1(\y_k) + \nabla_{\y}H(\x_u,\y_k), \G^{\ell}_{g}(\y_k) \big\rangle 
\\[1mm]
   {\geq}
   \beta_k  q_3  -k_3 (1-\beta_k).
\end{array}
   \]
   
\item for the lower problem: {setting} $q_4 \coloneqq \big\langle \partial g_2(\y_k) + \nabla_{\y}H(\x_{\ell},\y_k), \G^{\ell}_{g}(\y_k)\big\rangle$ and $k_4 \coloneqq (c_4+L_2)\big(\frac{1}{\nu}+L_2\big)$, 
   \[
   \begin{array}{lll}
   \hspace{-1.75cm}
   \big\langle \partial g_2(\y_k) + \nabla_{\y}H(\x_{\ell},\y_k), ~~  \beta_k \G^u_g(\y_k)  + (1-\beta_k) \G^{\ell}_{g}(\y_k) \big\rangle 
\\[1mm]
   = \beta_k 
   \big\langle \partial g_2(\y_k) + \nabla_{\y}H(\x_{\ell},\y_k), \G^u_g(\y_k)\big\rangle
   +(1-\beta_k) \big\langle \partial g_2(\y_k) + \nabla_{\y}H(\x_{\ell},\y_k), \G^{\ell}_{g}(\y_k) \big\rangle 
\\[1mm]
   {\geq} 
   -k_4 \beta_k + (1-\beta_k) q_4.
\end{array}
   \]
\end{itemize}
{Combining everything}, for $i = 1, \dots, 4$,
\begin{subequations}
\begin{equation}\label{conditions_alpha1}
{\alpha_k q_1  -k_1 (1-\alpha_k) > 0} 
\implies
0 \leq \frac{k_1}{q_1 + k_1}  \leq \alpha_k \leq 1;
\end{equation}
\begin{equation}\label{conditions_alpha2}
{{-k_2} \alpha_k  + (1-\alpha_k) q_2 > 0}  
~~~ \implies 
0 \leq  \alpha_k \leq \frac{q_2}{q_2 + k_2}  \leq 1;
\end{equation}
\begin{equation}\label{conditions_beta1}
{\beta_k  q_3  -k_3 (1-\beta_k) > 0} 
\implies
0 \leq \frac{k_3}{q_3 + k_3}  \leq \beta_k \leq 1.
\end{equation}
\begin{equation}\label{conditions_beta2}
{-k_4 \beta_k + (1-\beta_k) q_4 > 0}  
~~~  \implies
0 \leq \beta_k \leq \frac{q_4}{q_4 + k_4}  \leq 1.
\end{equation}
\end{subequations}
{This range for $\alpha_k, \beta_k$ yields} the descent conditions {of} the theorem.
\end{proof}

\section{Application to Sparse Low Rank Factorization}\label{sec:SLR}
In this section, we consider a sparse low rank Factorization (SLRF) problem \cite{sprechmann2015learning} {formulated as the bi-level problem} \eqref{eq:Binno}.
{Sparse low-rank factorization aims to decompose a data matrix $\M \in \IRmn$ into two factors: low-dimensional and interpretable. The low-rank structure is useful to capture global correlations in the data, while sparsity promotes localized representations, these properties are essential in real application tasks.
From a mathematical point of view, these two features are typically enforced through convex functions: the nuclear norm promotes low-rank solutions by working on singular values, whereas the $\ell_1$ norm, acts as a convex relaxation of the $\ell_0$-norm, inducing sparsity \cite{recht2010guaranteed,candes2006robust}. 
The resulting formulation combines these regularizations to balance structure and fidelity to the observed data.
Formally, in the setting of bi-level problem \eqref{eq:Binno},}
 we {aim to solve}:
\begin{equation}\label{eq:SLRF}
\hspace{-3mm}
\begin{array}{cl}
\displaystyle  
\argmin_{\X \in\IRmr,  \Y \in\IRrn} \hspace{-3mm} &\displaystyle    \lambda_1\|\X\|_1 + \lambda_2 \|\Y\|_1 + \tfrac{1}{2}\|\X\Y-\M\|_F^2 
\\
\st& \hspace{-9mm} \displaystyle
(\X, \Y) \in \hspace{-5mm}  \argmin_{\U \in\IRmr, \V \in\IRrn} \hspace{-5mm} 
\gamma_1 \|\U\|_* + \gamma_2 \|\V\|_*+ \tfrac{1}{2}\|\U\V-\M\|_F^2,
\end{array}
\tag{SLRF}
\end{equation}
where $\|\cdot\|_1$ is the elementwise $\ell_1$-norm, $\|\cdot\|_F$ {the Frobenius} norm and $\|\cdot\|_*$ the nuclear norm.
The functions in \eqref{eq:SLRF} {corresponding to those in}~\eqref{eq:Binno} are
$f_1:\IRmr\to\IR$, $f_1(\X)=\lambda_1\|\X\|_1$, 
$f_2:\IRmr\to\IR$, $f_2(\U)=\gamma_1\|\U\|_*$,
$g_1:\IRrn\to\IR$, $g_1(\Y)=\lambda_2\|\Y\|_1$,
$g_2:\IRrn\to\IR$, $g_2(\V)=\gamma_2\|\V\|_*$,
$H:\IRmr\times\IRrn\to\IR$, $H(\X,\Y)=\tfrac{1}{2}\|\X\Y-\M\|_F^2${, with $\lambda_i, \gamma_i$ for $i=1,2$ { regularization} hyperparameters. 
The motivation for adopting the bi-level formulation \eqref{eq:SLRF} is to exploit the hierarchical coupling between the low-rank approximation (lower level) and the reconstruction accuracy (upper level). Unlike existing bi-level schemes that require smoothness or convexity, \texttt{Binno} is specifically designed to handle the nonconvex and nonsmooth nature of real-world data, providing a stable descent condition to local stationary points where other methods might fail to provide meaningful decompositions.
}

We solve the problem with \texttt{Binno} {as follows}.

\paragraph{X-update}
We perform a ProxGrad step on $\X_k$ ({$\Y$ fixed at its most recent value}) {for} the upper level subproblem in \eqref{eq:SLRF},
$\X_u = \prox_{f_1}^\nu \big(
\X_k - \nu \nabla_{\X}H(\X_k,\Y)
\big)$.
{Then} a ProxGrad step on $\X_k$ 
({same} $\Y$) {for} the lower {level} subproblem, $\X_{\ell} = \prox_{f_2}^\nu \big( \X_k - \nu \nabla_{\X}H(\X_k,\Y)
\big)$.
{Finally, $\X_{k+1}$ is obtained by the} convex combination $\X_{k+1} = \alpha_k \X_u +(1-\alpha_k) \X_{\ell}$.
\paragraph{Y-update}
We perform a ProxGrad step on $\Y_k$ ({$\X$ fixed at} $\X_u$) {for} the upper level subproblem in \eqref{eq:SLRF},
$\Y_u = \prox_{g_1}^\nu \big(
\Y_k - \nu \nabla_{\Y}H(\X_u,\Y_k)
\big)$.
{Then} a ProxGrad step on $\Y_k$ ({using} $\X_{\ell}$ instead) {for} the lower level subproblem, $\Y_{\ell} = \prox_{g_2}^\nu \big(
\Y_k - \nu \nabla_{\Y}H(\X_{\ell},\Y_k)
\big)$.
{Finally, $\Y_{k+1}$ is obtained by the} convex combination $\Y_{k+1} = \beta_k \Y_u + (1-\beta_k) \Y_{\ell}$.

{Algorithm~\ref{alg:binno4slrf} details the implementation of \texttt{Binno} to problem~\eqref{eq:SLRF}.}

\begin{algorithm}[h!]
\label{alg:binno4slrf}
\DontPrintSemicolon
\textbf{Initialization}: $\X_0\in\IRmr;\Y_0\in\IRrn$
\\
\For{$k=1,2,...,MaxIter$}  
{ 
\textbf{Update for $\X$:} \\
\quad Upper-level update: $\X^u_{k} = \prox_{f_1}^\nu 
\big(
    \X_{k-1} - \nu \nabla_{\X}H(\X_{k-1},\Y_{k-1})
    \big)$
\\
\quad  Lower-level update: 
$\X^l_k = \prox_{f_2}^\nu 
\big(
    \X_{k-1} - \nu \nabla_{\X}H(\X_{k-1},\Y_{k-1})
    \big)$ 
\\
\quad Get $\alpha_k$ according to \cref{subsec:bounds_constant_alpha}
\\
\quad  Convex combination: 
$\X_{k}    =    \alpha_k \X^u_k    +(1-\alpha_k)\X^l_k$
\\
\textbf{Update for $\Y$:}
\\
\quad Upper-level update: 
$\Y^u_k = \prox_{g_1}^\nu 
\big(
    \Y_{k-1} - \nu \nabla_{\Y}H(\X^{u}_k,\Y_{k-1})
    \big)$ 
\\
\quad  Lower-level update: 
$\Y^l_k = \prox_{g_2}^\nu 
\big(
    \Y_{k-1} - \nu \nabla_{\Y}H(\X^{l}_k,\Y_{k-1})
    \big)$ 
\\
\quad Get $\beta_k$ according to \cref{bounds_constant_beta}
\\
\quad  Convex combination:
$\Y_{k}    =
    \beta_k \Y^u_k    +(1-\beta_k)\Y^l_k$
}
\KwOutput{$\X_{MaxIter}\in\IRmr$, $\Y_{MaxIter}\in\IRrn$}
\caption{\texttt{Binno} for \eqref{eq:SLRF} with input $\M\in\IRmn$, rank $r$}
\end{algorithm}

\subsection{Useful Theoretical Results}\label{subsec:useful_tools}
{The following results provide explicit bounds for the quantities appearing in the descent conditions of Theorem~\ref{th:descend_condition}. 
In particular, we control: (i) the size of subgradients associated with the regularizers, (ii) the deviation induced by proximal operators, and (iii) the smoothness constants of the coupling term $H$. 
These ingredients are combined later to derive computable bounds for the parameters $\alpha_k$ and $\beta_k$. To lighten the notation, in the following we omit the iteration index $k$. 
We begin by bounding the subgradients of the sparsity-inducing term.}
\begin{proposition}\label{prop:c1}
For $f_1:\IRmr\to\IR$ {the} element-wise $\ell_1$-norm, $\|\S\|_2 \leq \lambda_1\sqrt{mr}$ for any $\S \in \partial f_1$.
\begin{proof}
The subdifferential of $f_1(\X) = \lambda_1\|\X\|_1 = \lambda_1\sum_{ij} |x_{ij}|$ is
\[
\partial \|\X\|_1 = 
\left\{
\P \in \IRmr : p_{ij} \in 
\begin{cases}
\text{sign } x_{ij} & \text{if } x_{ij} \neq 0;
\\
p \in [-1,1] & \text{if } x_{ij} = 0 .
\end{cases}
\quad
\right\}.
\]
Let $\S \in \lambda_1 \partial \|\X\|_1$ be an element of the subdifferential.
Then
$
\|\S\|_2
\leq
\|\S\|_F 
= 
\sqrt{\textstyle\sum_{ij} s_{ij}^2} 
\leq \sqrt{\textstyle\sum_{ij} \lambda_1^2} 
= \lambda_1\sqrt{mr}
$.
\end{proof}
\end{proposition}
{We next consider the low-rank regularizer, whose subdifferential has a more structured form.}
\begin{proposition}\label{prop:c2}
For $f_2:\IRmr\to\IR$ {the} nuclear norm, $\|\S\|_2 \leq 2 \gamma_1$ for any $\S \in \partial f_2$.
\begin{proof}
The {subdifferential} of the nuclear norm $\|\X\|_* = \sum_{i=1}^r{\sigma_i(\X)}$ is  \cite{watson1992characterization}
\[
\partial\|\X\|_*
=
\big\{
\U\V^\top+\W ~\big|~ 
\W\in\IRmr, ~
\U^\top\W=\zeros, ~
\W\V=\zeros, ~
\|\W\|_2 \leq 1 
\big\},
\]
where $\X \overset{SVD}{=} \U\bSigma\V^\top$ with $k=\rk(\X)$, $\U\in\IR^{m\times k}$, $\bSigma=\Diag(\sigma_i(\X))\in\IR^{k\times k}$ and $\V\in\IR^{r\times k}$.

We show $\|\partial f_2(\X)\|_2\leq c_2$ in \cref{lem:subdiffbound}.
Let $\S \in \gamma_1 \partial \|\X\|_*$, so 
\[
\|\S\|_2 
= \gamma_1\|\U\V^\top+\W\|_2 
\leq \gamma_1 \left(\|\U\V^\top\|_2+\|\W\|_2\right) 
\leq \gamma_1(1 + 1) 
= 2\gamma_1.
\]
\end{proof}
\end{proposition}

{We now bound the deviation between a point and its proximal update, which will be used to control the proximal-gradient mappings.}

\begin{lemma}\label{lem:prox_f1}
For the elementwise matrix $\ell_1$-norm and a matrix $\X\in\IRmr$,
$
\big\|\X-\prox^{\nu}_{\lambda_1\|\cdot\|_1}(\X)\big\|_2 \leq \nu \lambda_1\sqrt{mr}
$.
\begin{proof}
The operator $\prox^{\nu}_{\lambda_1\|\cdot\|_1}$ is the soft-thresholding operator \cite{beck2017first}
\[
\big[\X-\prox^{\nu}_{\lambda_1\|\cdot\|_1}(\X)\big]_{ij} 
= 
\begin{cases}
\nu \lambda_1    & \hspace{-3mm} \text{if } \X_{ij} {>} \nu \lambda_1, \\
\X_{ij} & \hspace{-3mm} \text{if } |\X_{ij}| \leq \nu \lambda_1,\\
-\nu \lambda_1    & \hspace{-3mm} \text{if } \X_{ij} < -\nu \lambda_1.
\end{cases}
\]
{So}
$\big\|\X-\prox^{\nu}_{\lambda_1\|\cdot\|_1}(\X)\big\|_2
\leq
\big\|\X-\prox^{\nu}_{\lambda_1\|\cdot\|_1}(\X)\big\|_F 
\leq \hspace{-1mm} 
\sqrt{\sum_{ij} (\nu\lambda_1)^2} 
= \nu \lambda_1 \sqrt{mr}$.
\end{proof}
\end{lemma}

\begin{lemma}\label{lem:prox_f2}
For the nuclear norm,
$\big\|\X-\prox^{\nu }_{\gamma_1\|\cdot\|_*}(\X) \big\|_2  \leq  \gamma_1\nu$.
\begin{proof}
{For} the nuclear norm, $\prox^{\nu}_{\gamma_1\|\cdot\|_*}$ {reads}
\[
\prox_{\gamma_1\| \cdot \|_*}^{\nu}(\X) 
= 
\argmin_{\A}{\gamma_1\nu \|\A\|_*
+
\tfrac{1}{2}\|\A-\X\|^2}=\text{SVT}_{\gamma_1\nu}(\X)
\]
{which is} the Singular Value Thresholding (SVT) 
by the Von Neumann inequality \cite{cai2010singular}.
The SVT of $\X\overset{SVD}{=}\U\bSigma\V^\top$ is $\text{SVT}_{\gamma_1\nu}(\X)=\U\D_{\gamma_1\nu}(\bSigma)\V^\top$, with $\D_{\gamma_1\nu}(\bSigma)=\Diag\big([\sigma_i-\gamma_1\nu]_+\big)$ {the} soft-thresholded $\bSigma$ {where} $\sigma_i$ {are the singular values of} $\X$ and $[a]_+=\max\{a,0\}$.
Then, 
\[
\big\| \X - \text{SVT}_{\gamma_1\nu}(\X)  \big\|_2
= 
\big\|  \U  \big( \bSigma - \D_{\gamma_1\nu}(\bSigma) \big) \V^\top \big\|_2
= 
\| \bSigma - \D_{\gamma_1\nu}(\bSigma) \|_2
= 
\displaystyle \max_{i} \big| \sigma_i - [\sigma_i - \gamma_1\nu]_+ \big|.
\]
{Two cases}:
\begin{itemize}[leftmargin=15pt,itemsep=0pt]
\item \textbf{Case 1} $\sigma_i > \gamma_1\nu$:
$
\big| \sigma_i - [\sigma_i - \gamma_1\nu]_+ \big|
= \big| \sigma_i - (\sigma_i - \gamma_1\nu) \big|
= | \gamma_1\nu| = \gamma_1\nu
$.

\item \textbf{Case 2} $\sigma_i \leq \gamma_1\nu$:
$
\big| \sigma_i - [\sigma_i - \gamma_1\nu]_+ \big|
= \big| \sigma_i - 0 \big|
= | \sigma_i | = \sigma_i \leq \gamma_1\nu
$
.
\end{itemize}
Thus $\|\X-\text{SVT}_{\gamma_1\nu}(\X)\|_2\leq \gamma_1\nu$.
\end{proof}
\end{lemma}

{Finally, we characterize the smoothness of the data-fidelity term, which determines the Lipschitz constants entering the descent analysis.}

\begin{lemma}\label{lem:H_bismooth2}
For $H:\IRmr\times\IRrn\to\IR$ defined as $\frac{1}{2}\|\M-\X\Y\|_F^2$, the bi-smooth constant wrt $\X$ is $ L_1 = \|\Y\Y^\top \|_2$ and wrt $\Y$ is $ L_2 = \|\X^\top\X \|_2$.
\begin{proof}
{We have} $ L_1 = \|\nabla_{\X}^2 H(\X, \Y)\| = \|\Y\Y^\top \|_2$, the spectral norm of $\Y\Y^\top$,
and $L_2 = \|\nabla_{\Y}^2 H(\X, \Y)\| = \|\X^\top\X \|_2$.
\end{proof}
\end{lemma}

\subsection{Finding the constant $\alpha_k$}\label{subsec:bounds_constant_alpha}
Here, the results are {stated} for any $\X$ and $\Y$ according to the iteration structure in Algorithm~\ref{alg:binno4slrf}.
To find {the} constants in \eqref{conditions_alpha1} and \eqref{conditions_alpha2},
we need
\[
\begin{array}{lll}
    k_1 = (c_1+L_1)\big(\tfrac{1}{\nu}+L_1\big);
    &&
    q_1 =  \big\langle \partial f_1(\X) + \nabla_{\X}H(\X,\Y), \G^{u}_f(\X)\big\rangle;
    \vspace{1mm}
    \\
    k_2 = (c_2+L_1)\big(\tfrac{1}{\nu}+L_1\big);  
    &&
    q_2 = \big\langle \partial f_2(\X) + \nabla_{\X}H(\X,\Y), \G^{\ell}_f(\X) \big\rangle;
\end{array}
\]
where $c_1,c_2$ are as in Lemma~\ref{lem:subdiffbound} for $f_1$ and $f_2$ respectively, and $L_1$ is the bi-smooth constant for $H$ wrt $\X$.
{In particular,} $c_1= \lambda_1\sqrt{mr}$ {by} \cref{prop:c1}, $c_2=2\gamma_1$ {by} \cref{prop:c2}, and $L_1 = \|\Y\Y^\top\|_2$ {by} \cref{lem:H_bismooth2}.
{Bounds} for $q_1$ and $q_2$ {are given by the following proposition}.
\begin{proposition}\label{prop:l1}
For $q_1 = \langle \partial f_1 (\X) + \nabla_{\X} H(\X, \Y), \G^{u}_f(\X) \rangle$, we have that 
$q_1 \leq (\lambda_1\sqrt{mr}+\|\Y\Y^\top \|_2)^2$.
For $q_2$ = $\langle \partial f_2 (\X) + \nabla_{\X} H(\X, \Y), \G^{\ell}_f(\X) \rangle$, we have that $ q_2 \leq (\gamma_1+\|\Y\Y^\top\|_2)(2\gamma_1+\|\Y\Y^\top\|_2)$.
\begin{proof}
\[
\begin{array}{rcl}
q_1
&\leq&
\big| \langle \partial f_1 (\X), \G^{u}_f(\X) \rangle\big| + \big| \langle \nabla_{\X} H( \X, \Y), \G^{u}_f(\X) \rangle\big|
\\
&\overset{\cref{lem:subdiffbound},\cref{lem:H_bismooth}}{\leq}&
(c_1 + L_1 ) \|\G^{u}_f(\X)\|_2
\\
&\overset{\cref{lem:g}}{\leq}&
(c_1+L_1)\Big(\tfrac{1}{\nu} \big\|\X-\prox^{\nu }_{\lambda_1\|\cdot\|_1}(\X)\big\|_2 + L_1\Big)
\\
&\overset{\cref{lem:prox_f1}}{\leq}&
(c_1 + L_1)(\lambda_1\sqrt{mr}+L_1) 
\\
&\overset{\cref{prop:c1}}{=}&
(c_1 + L_1)^2
\overset{\cref{lem:H_bismooth2}}{=}
\big(\lambda_1\sqrt{mr}+\|\Y\Y^\top \|_2 \big)^2.
\end{array}
\]
\[
\begin{array}{rcl}
q_2 
&\leq&
| \langle \partial f_2 (\X), \G^{\ell}_f(\X) \rangle| + | \langle \nabla_{\X} H( \X, \Y), \G^{\ell}_f(\X) \rangle|
\\
&\overset{\cref{lem:subdiffbound},\cref{lem:H_bismooth}}{\leq}&
(c_2+ L_1 ) \|\G^{\ell}_f(\X)\|_2
\\
&\overset{\cref{lem:g}}{\leq}&
(c_2+L_1)\Big(\tfrac{1}{\nu}
\big\|\X-\prox^{\nu}_{\gamma_1\|\cdot\|_*}(\X)\big\|_2 + L_1\Big)
\\
&\overset{\cref{lem:prox_f2}}{\leq}& 
(c_2 + L_1)(\gamma_1+L_1) 
\overset{\cref{prop:c2}}{=}
(2\gamma_1+L_1)(\gamma_1+L_1)
\\
&\overset{\cref{lem:H_bismooth2}}{=}&
(\gamma_1+\|\Y\Y^\top\|_2)(2\gamma_1+\|\Y\Y^\top\|_2).
\end{array}
\]
\end{proof}
\end{proposition}

{To guarantee that the descent conditions admit feasible parameters, we need to ensure that the interval defined by \eqref{conditions_alpha1}--\eqref{conditions_alpha2} is non-empty. The following proposition serves as sufficient condition for this aim.}
{
\begin{proposition}[Non-empty Interval]\label{prop:feas}
Let $c,L,L',\mu>0$.
The interval
\[
\left[
\frac{\tfrac{1}{\nu}+L}{c+2L+\tfrac{1}{\nu}}, ~ \frac{\mu+L'}{\mu+2L'+\tfrac{1}{\nu}}
\right]
~\subseteq~ [0,1] 
\]
is non-empty iff $\nu \geq \tfrac{1}{\sqrt{(L+\mu)(L+c)}-L}$ 
(up to the generalization when $L$ differs between numerator and denominator).
\end{proposition}
\begin{proof}
First we show the interval sits in $[0,1]$, so non-emptiness is the only issue. 
Write $z := 1/\nu > 0$.

\emph{Endpoints lie in $[0,1]$.} Both numerators and denominators are positive, so the endpoints are positive. For the upper bound of 1:
\[
\frac{z+L}{c+2L+z} \leq 1 \iff z+L \leq c+2L+z \iff 0 \leq c+L, 
\]
\[
\frac{\mu+L'}{\mu+2L'+z} \leq 1 \iff \mu+L' \leq \mu+2L'+z \iff 0 \leq L'+z,
\]
both of which hold.

\emph{Non-emptiness.} The lower endpoint does not exceed the upper iff
\[
\frac{z+L}{c+2L+z} \;\leq\; \frac{\mu+L'}{\mu+2L'+z}.
\]
Cross-multiplying (denominators positive) and expanding, the terms $z\mu$ and $2LL'$ cancel across sides and we obtain
\begin{equation}\label{eq:feas_quadratic}
z^2 + (L+L')\,z - \big(L\mu + \mu c + L'c\big) \leq 0.
\end{equation}
This quadratic in $z$ has discriminant $(L+L')^2 + 4(L\mu+\mu c+L'c) > 0$, hence a unique positive root
\[
z^\star = \tfrac{1}{2}\Big[\sqrt{(L+L')^2 + 4(L\mu+\mu c+L'c)} - (L+L')\Big],
\]
and \eqref{eq:feas_quadratic} holds iff $z \leq z^\star$, i.e.\ $\nu \geq 1/z^\star$.

When $L=L'$, by the identity $L\mu + \mu c + Lc = (L+\mu)(L+c) - L^2$, rewrite \eqref{eq:feas_quadratic} as $(z+L)^2 \leq (L+\mu)(L+c)$, giving $z \leq \sqrt{(L+\mu)(L+c)} - L$, i.e.\ $\nu \geq 1/\big(\sqrt{(L+\mu)(L+c)} - L\big)$.
\end{proof}

Now we make use of Non-empty Interval to show that stepsizes for \texttt{Binno} exist.

\begin{theorem}\label{th:alpha_range}
Let the assumptions of Theorem~\ref{th:descend_condition} hold and consider Problem~\eqref{eq:SLRF}. 
If the stepsize $\nu$ satisfies
\[
\nu \;\geq\; \dfrac{1}{\sqrt{(\|\Y\Y^\top\|_2+\gamma_1)(\|\Y\Y^\top\|_2+\lambda_1\sqrt{mr})} - \|\Y\Y^\top\|_2},
\]
then the set of $\alpha_k\in[0,1]$ satisfying conditions \eqref{conditions_alpha1}--\eqref{conditions_alpha2} is non-empty, and every such $\alpha_k$ lies in
\[
0\leq \dfrac{\tfrac{1}{\nu}+\|\Y\Y^\top\|_2}{\lambda_1\sqrt{mr}+\tfrac{1}{\nu}+2\|\Y\Y^\top\|_2}\leq \alpha_k\leq 
\dfrac{\gamma_1+\|\Y\Y^\top\|_2}{\gamma_1+\tfrac{1}{\nu}+2\|\Y\Y^\top\|_2}
\leq 1.
\]
\end{theorem}
}

\begin{proof}
{Conditions \eqref{conditions_alpha1}--\eqref{conditions_alpha2} give into $k_1/(q_1+k_1) \leq \alpha_k \leq q_2/(q_2+k_2)$. Substituting the expressions for $q_i,k_i$ from Section~\ref{subsec:bounds_constant_alpha} and Proposition~\ref{prop:l1} gives exactly the claimed interval, with
\[
c = \lambda_1\sqrt{mr},~~ \mu = \gamma_1,~~ L = L' = \|\Y\Y^\top\|_2
\]
in the notation of Proposition~\ref{prop:feas}.}
\end{proof}

{
\begin{remark}{(Numerical stability)}
The denominator $\sqrt{(a+c)(a+b)}-a$ in Theorem~\ref{th:alpha_range} can suffer from catastrophic cancellation when $a$ dominates. 
To prevent it, we consider 
$\tfrac{1}{\sqrt{(a+c)(a+b)}-a} = \tfrac{\sqrt{(a+c)(a+b)}+a}{ab+ac+bc}$
which we use in the implementation. 
\end{remark}
}

\subsection{Finding the constant $\beta_k$}\label{bounds_constant_beta}
Similar to $\alpha_k$, {we derive the following for} $\beta_k$.
To find {the} constants in \eqref{conditions_beta1} and \eqref{conditions_beta2},
we need
\[
\begin{array}{lll}
    k_3 = \big( c_3+L_2(\X_u) \big)\big(\tfrac{1}{\nu}+L_2(\X_u)\big), 
    &\quad
    q_3 =   \big\langle \partial g_1(\Y) + \nabla_{\Y}H(\X_u,\Y), \G^u_g(\Y)\big\rangle;   
    \\
    k_4 = \big(c_4+L_2(\X_{\ell})\big)\big(\tfrac{1}{\nu}+L_2(\X_{\ell})\big),
    &\quad
    q_4 =  \big\langle \partial g_2(\Y) + \nabla_{\Y}H(\X_{\ell},\Y), \G^{\ell}_{g}(\Y)\big\rangle;
\end{array}
\]
where $c_3,c_4$ {are the} constants for $g_1$ and $g_2$ respectively as in \cref{lem:subdiffbound}, and $L_2(\X_{\Delta})$ is the bi-smooth constant for $H$ wrt $\Y$ computed {at} $\X_{\Delta}\in\{\X_u,\X_{\ell}\}$.
{In particular,} $c_3=\lambda_2\sqrt{rn}$ by \cref{prop:c1}, $c_4=2\gamma_2$ by \cref{prop:c2}, and $L_2 = \|\X_\Delta\X_\Delta^\top\|_2$ by \cref{lem:H_bismooth2}.
{We drop subscripts where unnecessary, to avoid confusion}.

{We now bound $q_3$ and $q_4$}.
\begin{proposition}\label{prop:l3}
For $q_3$ = $\big\langle \partial g_1(\Y) + \nabla_{\Y}H(\X_u,\Y), \G^u_g(\Y)\big\rangle$,
we have that 
$q_3 \leq (\lambda_2\sqrt{rn}+\|\X_u^\top\X_u \|_2)^2$.
For $q_4$ = $\langle \partial g_2 (\Y) + \nabla_{\Y} H(\X_{\ell}, \Y), \G^{\ell}_{g}(\Y) \rangle$,
we have that 
$q_4 \leq (2\gamma_2+\|\X_{\ell}^{\top}\X_{\ell}\|_2)(\gamma_2+\|\X_{\ell}^{\top}\X_{\ell}\|_2)$.
\begin{proof}
\[
\begin{array}{lcl}
q_3 
&\leq&
\big| \big\langle \partial g_1 (\Y), \G^u_g(\Y) \big\rangle \big| + \big| \big\langle \nabla_Y H( \X_u, \Y), \G^u_g(\Y) \big\rangle\big|
\\
&\overset{\cref{lem:subdiffbound},\cref{lem:H_bismooth2}}{\leq}&
(c_3 + L_2(\X_u)) \|\G_u (\Y_k)\|_2
\\
&\overset{\cref{lem:g}}{\leq}&
(c_3+L_2(\X_u))
\big(\tfrac{1}{\nu} \big\|\Y_k-\prox^{\nu}_{\|\cdot\|_1}(\Y_k)\big\|_2 + L_2(\X_u)\big)
\\
&\overset{\cref{lem:prox_f1}}{\leq}& (c_3 + L_2(\X_u))(\lambda_2\sqrt{rn}+L_2(\X_u)) 
\\
&\overset{\cref{prop:c1}}{=}&
(c_3 + L_2(\X_u))^2
\overset{\cref{lem:H_bismooth2}}{\leq} (\lambda_2\sqrt{rn}+\|\X_u^\top\X_u \|_2)^2.
\end{array}
\]
\[
\begin{array}{lcl}
q_4 
&\leq&
\big| \big\langle \partial g_2 (\Y), \G^{\ell}_{g}(\Y)  \big\rangle\big| + \big| \big\langle \nabla_Y H( \X_{\ell}, \Y), \G^{\ell}_{g}(\Y)  \big\rangle\big|
\\
&\overset{\cref{lem:subdiffbound},\cref{lem:H_bismooth2}}{\leq}&
(c_4 + L_2(\X_{\ell})) \|\G^{\ell}_{g}(\Y) \|_2
\\
&\overset{\cref{lem:g}}{\leq}&
(c_4+L_2(\X_{\ell}))\big(\tfrac{1}{\nu}
\big\|\Y-\prox^{\nu}_{\|\cdot\|_*}(\Y)\big\|_2 + L_2(\X_{\ell})\big)
\\
&\overset{\cref{lem:prox_f2}}{\leq}& 
(c_4 + L_2(\X_{\ell}))(\gamma_2 + L_2(\X_{\ell}))
\\
&\overset{\cref{prop:c2}}{\leq}&
(2\gamma_2 +L_2(\X_{\ell}))
(\gamma_2+L_2(\X_{\ell}))
\\
&\overset{\cref{lem:H_bismooth2}}{=}&
(2\gamma_2 + \|\X_{\ell}^\top\X_{\ell}\|_2)
(\gamma_2+ \|\X_{\ell}^\top\X_{\ell}\|_2).
\end{array}
\]
\end{proof}
\end{proposition}

{
\begin{theorem}\label{th:beta_range}
Let the assumptions of Theorem~\ref{th:descend_condition} hold and consider Problem~\eqref{eq:SLRF}.
Let $N = \|\X_{\ell}^\top\X_{\ell}\|_2+\|\X_u^\top\X_u\|_2$. 
If the stepsize $\nu$ satisfies
\[
\nu
\;\geq\; 
\dfrac{2}{
\sqrt{N^2
+4\big(\lambda_2\gamma_2\sqrt{rn}+\gamma_2\|\X_u^\top\X_u\|_2+\|\X_{\ell}^\top\X_{\ell}\|_2\lambda_2\sqrt{rn}\big)}-N},
\]
then the set of $\beta_k\in[0,1]$ satisfying conditions \eqref{conditions_beta1}--\eqref{conditions_beta2} is non-empty, and every such $\beta_k$ lies in
\[
0\leq \dfrac{\tfrac{1}{\nu}+\|\X_u^\top\X_u\|_2}{\lambda_2\sqrt{rn}+\tfrac{1}{\nu}+2\|\X_u^\top\X_u\|_2}\leq \beta_k\leq 
\dfrac{\gamma_2+\|\X_{\ell}^\top\X_{\ell}\|_2}{\gamma_2+\tfrac{1}{\nu}+2\|\X_{\ell}^\top\X_{\ell}\|_2}\leq 1.
\]
\end{theorem}
}
\begin{proof}
{
Conditions \eqref{conditions_beta1}--\eqref{conditions_beta2} give into $k_3/(q_3+k_3) \leq \beta_k \leq q_4/(q_4+k_4)$. Substituting the expressions for $q_i,k_i$ from Section~\ref{bounds_constant_beta} and Proposition~\ref{prop:l3} gives the claimed interval, with
\[
c = \lambda_2\sqrt{rn},~ \mu=\gamma_2,~ L=\|\X_u^\top\X_u\|_2,~ L'=\|\X_{\ell}^\top\X_{\ell}\|_2
\]
in the notation of Proposition~\ref{prop:feas}. 
Applying the case ($L\neq L'$) with $L+L'=N$ yields the stated lower bound on $\nu$.
}
\end{proof}

{\begin{remark}
The hyperparameters ($\lambda_1,\lambda_2,\gamma_1,\gamma_2$) weight the regularization terms {associated with the constraints on} the factor matrices, and {thus} control the trade-off between data fidelity and structural regularization; setting all {to} $1$ would {generally be} arbitrary, since the corresponding terms may have different roles and scales. {Clearly}, hyperparameter tuning is a crucial issue and deserves careful consideration. In principle, this step could also be addressed {via} nested bi-level or, more generally, multi-level optimization strategies \cite{Selicato2025,SELICATO2026117316}.
\end{remark}}

{
\section{Experimental Validation of \texttt{Binno} on SLRF}\label{sec:exp}
We evaluate \texttt{Binno} on synthetic and real data, assessing accuracy, 
fidelity, and efficiency.\footnote{All experiments run in MATLAB~2024b on an Intel~i7 
octa-core workstation with 16GB RAM.
Code available at 
\url{https://github.com/flaespo/Binno.git}.}
Since problem~\eqref{eq:SLRF} is nonconvex, all compared methods can only 
be expected to reach local optima; the evaluation therefore concerns the 
quality of the recovered solutions rather than global optimality.

\paragraph{Baselines}
We compare \texttt{Binno} against two kinds of baselines:
\begin{itemize}[leftmargin=15pt,itemsep=0pt]
    \item \emph{Task-specific SLRF solvers.} \texttt{NMFLS}\cite{Ji:2013:EMNLP}, 
    a nonnegative matrix factorization routine based on multiplicative updates~\cite{Lee1999}; and 
    \texttt{NSA}\cite{NSA}, an ADMM scheme~\cite{Bauschke2017,lin2022alternating} for robust principal component pursuit, combining a partial SVD for the 
    low-rank update with entrywise soft-thresholding for the sparse update. 
    We use two implementations, \texttt{NSA-v1} and \texttt{NSA-v2}, which 
    share the core iteration and stopping rule but differ in parameterization 
    and in the warm start of the partial 
    SVD~\cite{AybatGoldfarbMa2014,AybatIyengar2015}.
    \item \emph{Bi-level baseline.} \texttt{BiG-SAM}\cite{sabach2017first}, 
    a 1st-order bi-level method for smooth-plus-convex problems, adapted 
    here to~\eqref{eq:SLRF}.
\end{itemize}
The choice of baselines is deliberate: \texttt{NMFLS} and \texttt{NSA} are 
established solvers for sparse-plus-low-rank decomposition, while 
\texttt{BiG-SAM} is the closest available bi-level method. 
Most other 1st-order bi-level solvers require a smooth or strongly convex upper level and thus do not apply to~\eqref{eq:SLRF} without reformulation. 
Implementations of \texttt{NMFLS} and \texttt{NSA} are taken from 
LRSLibrary\cite{lrslibrary2015}.\footnote{%
\url{https://github.com/andrewssobral/lrslibrary/tree/master/algorithms}}

\paragraph{Datasets} The synthetic dataset is generated
$\M = \X_\star \Y_\star + \N \in \IR^{m\times n}$, where 
$\X_\star \in \IR^{m\times r}$ and $\Y_\star \in \IR^{r\times n}$ have i.i.d.\ 
standard normal entries on a sparse support, and $\N$ is entrywise Gaussian 
with standard deviation $\sigma = 0.01$. Unless stated otherwise, $m = 100$ 
and $n = 80$. The rank $r$ is treated as a design parameter: in 
\cref{subsubsec:synth_fixed} we fix $r = 5$ with $\approx 30\%$ nonzero 
entries in the factors; in \cref{subsubsec:synth_sweep} we vary 
$r \in [2,30]$ with $10$ random initializations per rank to probe robustness.

For real data, we use the Seattle traffic-video 
database~\cite{1467355,chan2005classification}, consisting of 254 video sequences of highway traffic from a single stationary camera (42--52 frames per clip at 10fps)\footnote{%
\url{https://github.com/andrewssobral/lrslibrary/tree/master/dataset/trafficdb}}. Each clip is converted to grayscale, resized to $80\times 60$ pixels, and cropped to a $48\times 48$ window over the most active region. The fixed viewpoint makes the sequences suitable for background-foreground separation via low-rank plus sparse modeling, with $r = 5$ for all clips.

\paragraph{Setup}
All methods share the same initialization $(\X_0,\Y_0)$ drawn entrywise 
from a standard Gaussian, and terminate once a fixed maximum number of iterations is reached\footnote{Ee.g. 1000 on synthetic data.}. 
For \texttt{Binno} we set $\lambda_1 = \lambda_2 = 10^{-2}$ at the upper level 
and $\gamma_1 = \gamma_2 = 5{\cdot}10^{-2}$ at the lower level.
These  coefficients weight the $\ell_1$ and nuclear-norm penalties respectively, 
which act on blocks of different numerical scale; setting them equal across 
levels is therefore not a neutral choice. The values above were obtained by 
a coarse grid search and produced stable recovery across all ranks tested; 
we observed that \texttt{Binno}'s output is robust to perturbations within one order of magnitude around these values. 
Baselines use their published default 
parameters.

\paragraph{Metrics}
We report the relative reconstruction error, the peak signal-to-noise ratio 
(PSNR), and the CPU time. Given the iterate $(\X,\Y)$, the recovered 
low-rank component is $\L = \X\Y$ and the sparse component, capturing 
structured noise or moving foreground, is the residual $\S = \M - \L$. 
The Frobenius-norm relative error $\lVert \M - \L \rVert_F / \lVert \M \rVert_F$
is standard in low-rank modeling and RPCA~\cite{HalkoMartinssonTropp2011}. 
PSNR is computed from the mean-squared error 
between reference and reconstruction as
$\mathrm{PSNR} = 10 \log_{10} \!\left(\mathrm{MAX}^2 /\mathrm{MSE}\right)$,
where $\mathrm{MAX}$ is the peak representable value in the image~\cite{WangBovik2006}.

\subsection{Results on synthetic data}\label{subsec:res_synth}

\subsubsection{Fixed-rank comparison against task-specific solvers}
\label{subsubsec:synth_fixed}
We perform a simple toy testing with rank $r=5$. 
Results in 
\cref{tab:synt} shows that \texttt{Binno} attains the lowest reconstruction error by 
roughly an order of magnitude at a competitive runtime.
\begin{table*}[h!]
\caption{Evaluated metrics on synthetic data, $r=5$.}
\label{tab:synt}
\centering
\begin{tabular}{r|cccc}
                & \texttt{Binno} & \texttt{NMFLS} & \texttt{NSA-v1} & \texttt{NSA-v2} \\
\hline
Time            & $0.093$        & $0.102$        & $0.169$         & $0.065$         \\
Reconstruction error 
                & $0.0135$       & $1.0412$       & $0.3259$        & $0.3259$
\end{tabular}
\end{table*}

\subsubsection{Rank sensitivity against the bi-level baseline}
\label{subsubsec:synth_sweep}
We now assess robustness to the target rank by varying $r \in [2,30]$ with 
$10$ random initializations per rank, and compare \texttt{Binno} with 
\texttt{BiG-SAM}.

\begin{remark}
\texttt{BiG-SAM} is designed for smooth-plus-convex bi-level problems 
(\cref{sec:background}), a setting strictly less general than SLRF. 
The comparison is intended to quantify the effect of \texttt{Binno}'s 
coupled descent mechanism relative to the closest applicable bi-level 
baseline, rather than to claim that \texttt{BiG-SAM} is the appropriate 
solver for this nonconvex problem.
\end{remark}

\begin{figure*}[h!]
\centering\includegraphics[width=\textwidth]{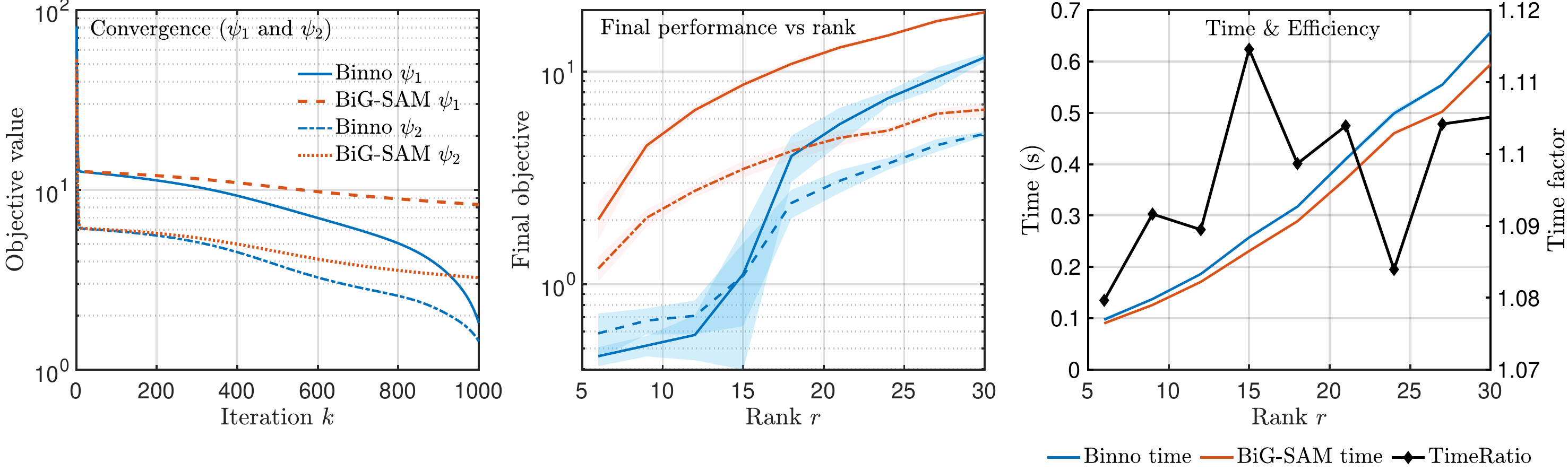}
\caption{\textbf{Left:} convergence history on a single instance at $r=15$, 
showing the joint minimization of the upper-level ($\psi_1$) and lower-level 
($\psi_2$) objectives. \textbf{Middle:} mean performance over 10 random 
trials per rank; shaded regions denote $\pm 1$ standard deviation. 
\textbf{Right:} dual-axis plot of absolute CPU time and the time factor 
of \texttt{Binno} over \texttt{BiG-SAM}.}
\label{fig:vsBigSam}
\end{figure*}

As shown in Fig.~\ref{fig:vsBigSam} (left), \texttt{Binno} exhibits a markedly 
steeper simultaneous descent of $\psi_1$ and $\psi_2$; the upper-level 
objective, which governs sparsity of the factors, plateaus substantially 
earlier than under \texttt{BiG-SAM}. The statistical summary (middle) 
confirms that this advantage is consistent across all ranks, with narrow 
variance bands indicating that \texttt{Binno} is robust to random 
initialization, whereas \texttt{BiG-SAM} displays higher variance and  
degrades in accuracy as $r$ grows. 
Both methods scale roughly linearly in 
$r$, but \texttt{Binno} incurs only a tiny runtime overhead (right), due to the additional scalar computations for \eqref{conditions_alpha1}-\eqref{conditions_beta2}. It should be noted that, given the nonconvex nature of the problem, the solutions reached by the compared solvers are local optima. The fact that \texttt{Binno} achieves a lower Relative Reconstruction Error suggests that its descent-driven coupling is particularly effective at navigating the nonsmooth landscape of matrix factorization, avoiding some of the poor local minima where standard methods might stagnate.

\subsection{Results on real-world video}\label{subsec:res_real}
Fig.~\ref{fig:qualitative} reports the convergence of $\psi_1$, $\psi_2$ on a 
representative clip; as in the synthetic case, both objectives decrease 
monotonically, consistent with the descent safeguards built into the algorithm (proximal-gradient blocks and calibrated averaging).

Fig.~\ref{fig:qualitative} shows an original clip with its noisy observation 
together with the recovered low-rank component $\L$ and sparse component 
$\S$ for each method; Fig.~\ref{fig:qualitative6} displays the corresponding 
decompositions for six additional clips spanning diverse dynamics. 
Qualitatively, all methods produce visually comparable decompositions.

\begin{figure}[h!]
\centering\includegraphics[width=0.37\textwidth]{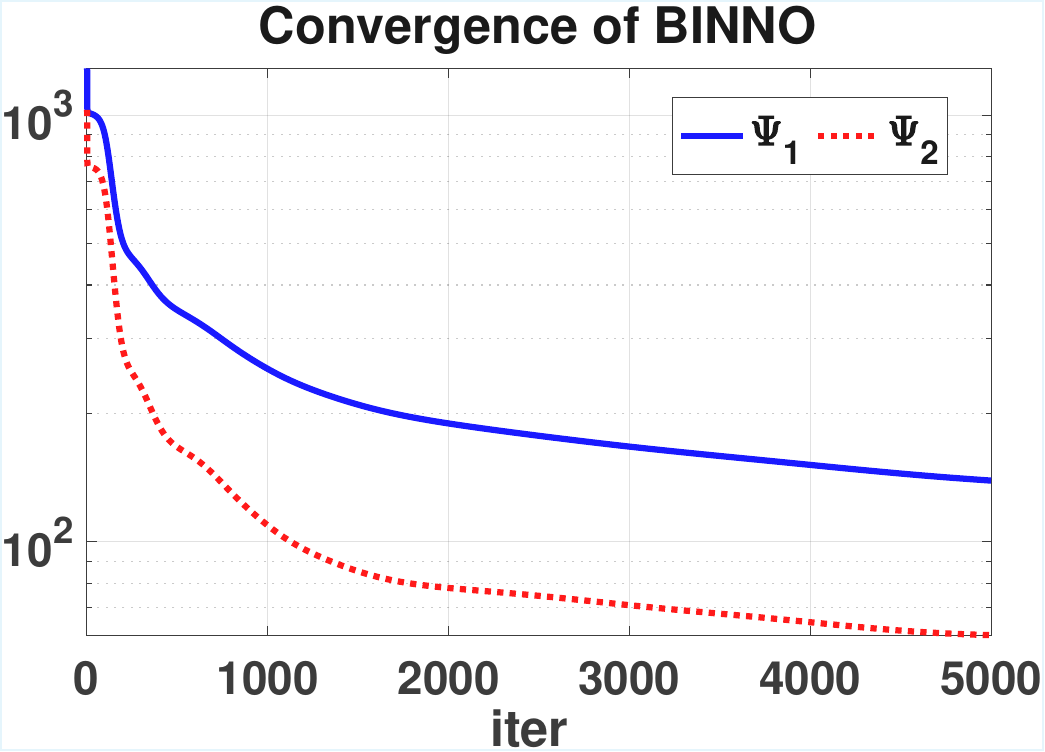}
\includegraphics[width=0.4\textwidth]{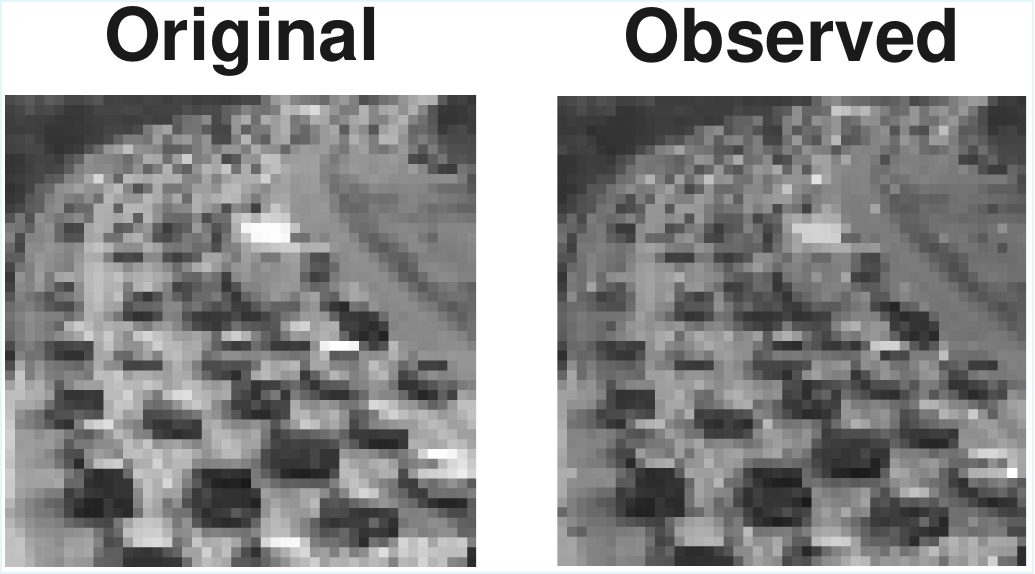}\\
\includegraphics[width=0.66\textwidth]{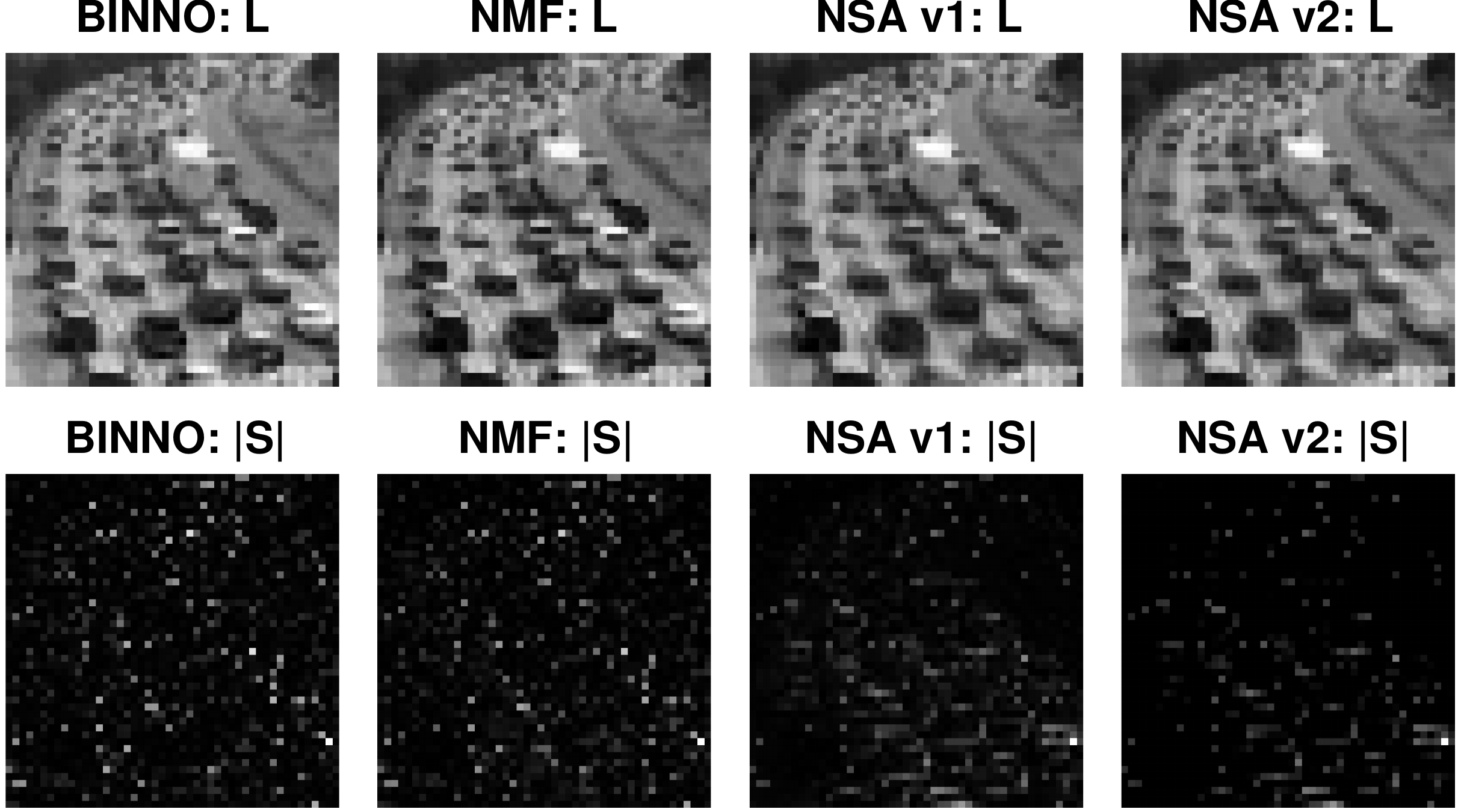}
\caption{
Top row: Convergence of the bi-level objectives on a representative video clip, and a representative clip (clip 5) with its noisy observation; 
Right Bottom two rows: low-rank $\L$ and sparse $\S$ components recovered by each method.}
\label{fig:qualitative}
\end{figure}

\begin{figure}[h!]
\centering\includegraphics[width=0.66\textwidth]{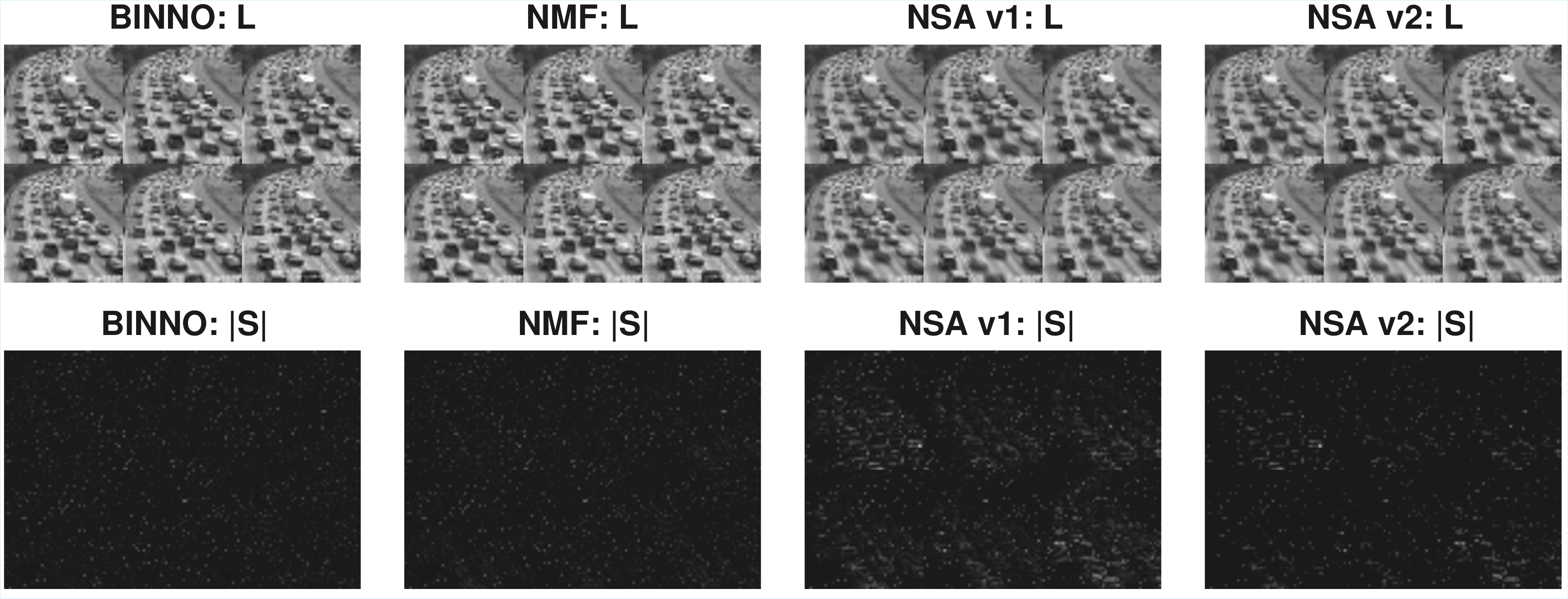}
\caption{Recovered $\L$ and $\S$ for all methods on six additional clips.}
\label{fig:qualitative6}
\end{figure}

Quantitative results are summarized in \cref{tab:comparison} as 
mean$\pm$standard deviation over the 254 clips. \texttt{Binno} attains 
the lowest reconstruction error and the highest PSNR on every evaluated 
sequence, at the cost of a moderate runtime overhead.

\begin{table*}[h!]
\caption{Mean$\pm$std of the evaluated metrics over all 254 video clips.}
\label{tab:comparison}
\centering
\small
\begin{tabular}{c|cccc}
& \texttt{Binno} & \texttt{NMFLS} & \texttt{NSA-v1} & \texttt{NSA-v2} \\
\hline
Time 
& $4.092\!\pm\!0.3615$ 
& $3.0769\!\pm\!1.4457$ 
& $0.1754\!\pm\!0.05$ 
& $0.1903\!\pm\!0.05$ \\
Err 
& $0.0663\!\pm\!0.0012$ 
& $0.0680\!\pm\!0.0025$ 
& $0.0982\!\pm\!0.0094$ 
& $0.0982\!\pm\!0.0094$ \\
PSNR 
& $35.16\!\pm\!1.36$ 
& $34.27\!\pm\!1.34$ 
& $28.43\!\pm\!1.53$ 
& $28.43\!\pm\!1.53$
\end{tabular}
\end{table*}
}

{
\section{Application to Regularized Market Clearing}\label{sec:RMC}
As bi-level optimization are closely related to two-player games, we apply \texttt{Binno} to economics of regularized market clearing (RMC): the lower level fixes the set of market equilibria, and the upper level picks, among them, the one that best meets declared policy objectives.
RMC fits \eqref{eq:Binno} exactly and is distinct from SLRF.

\paragraph{Model}
A market has $m$ goods, $r$ sellers (player 1), and $n$ buyers (player 2). 
Let $\X\in\IR^{m\times r}$ be seller decisions (supply intensities across goods), $\Y\in\IR^{r\times n}$ be buyer allocations (how each seller serves each buyer), and $\M\in\IR^{m\times n}$ be the clearing target (demand profile).
Then $\X\Y$ is the aggregate market outcome, and $H(\X,\Y) = \tfrac{1}{2}\|\X\Y - \M\|_F^2$ is market-clearing residual. 
\begin{itemize}[leftmargin=15pt,itemsep=0pt]
\item Lower level (auction market behaviour): sellers and buyers reach an equilibrium characterized by sparse bidding (each seller supplies few goods, each buyer is matched to few sellers).
\item Upper level (social wealth maximization by policy): the regulator selects among such equilibria a policy-preferred outcome: a \emph{concentrated} (low-rank) production structure on $\X$, to ease regulatory oversight, and a \emph{sparse} allocation $\Y$, for traceability. 
\end{itemize}
RMC written in the form of \eqref{eq:Binno} is thus
\begin{equation}\label{eq:RMC}
\hspace{-5mm}
\begin{array}{cl}
\displaystyle\argmin_{\X\in\IR^{m\times r},\,\Y\in\IR^{r\times n}} \hspace{-3mm}
& \displaystyle \lambda_1\|\X\|_* + \lambda_2\|\Y\|_1 + \tfrac{1}{2}\|\X\Y-\M\|_F^2
\\[1mm]
\st & \hspace{-9mm}\displaystyle
(\X,\Y) \in \hspace{-5mm}\argmin_{\U\in\IR^{m\times r},\,\V\in\IR^{r\times n}}\hspace{-5mm}
\mu_1 \|\U\|_1 + \mu_2 \|\V\|_1 + \tfrac{1}{2}\|\U\V - \M\|_F^2,
\end{array}
\tag{RMC}
\end{equation}
where $\lambda_1,\lambda_2,\mu_1,\mu_2>0$ weight policy stringency and equilibrium regularity, respectively. Compared to \eqref{eq:SLRF}, the lower level here is doubly sparse (two $\ell_1$ blocks), and the policy term $f_1$ promotes low-rank structure instead of sparsity.

\paragraph{Explainability in economics}
RMC means that, among all equilibria in which supply and allocation decisions are sparse, the regulator prefers those in which the seller matrix is explainable by a small number of production archetypes ($\|\X\|_*$ small) and in which each buyer is served by few sellers ($\|\Y\|_1$ small).

\subsection{Numerical experiments}\label{subsec:res_RMC}

\paragraph{Data and Baselines}
We use UCI Online Retail dataset\footnote{\url{https://archive.ics.uci.edu/dataset/352/online+retail}} (on UK retailer transactions during 2010-2011).
We compare \texttt{Binno} (solving \eqref{eq:RMC} directly) to \texttt{BiG-SAM} (adapted as in Section~\ref{subsubsec:synth_sweep}), and also \texttt{Lower-only}: PALM applied to the inner problem alone, returning an arbitrary equilibrium (dependent on initialization). We run it from $10$ random starts to illustrate non-uniqueness.

\paragraph{Metrics}
Besides the clearing residual $\|\X\Y-\M\|_F/\|\M\|_F$, we evaluate \emph{policy quality} through
\begin{itemize}[leftmargin=15pt,itemsep=0pt]
\item \emph{Sparsity} of $\Y$: fraction of entries below $10^{-3}$ in magnitude.
\item \emph{Policy score} $\mathcal{P} = \lambda_1\|\X\|_* + \lambda_2\|\Y\|_1$, which is the upper-level objective restricted to the regularization part.
\end{itemize}
We set $\lambda_1=\lambda_2=10^{-2}$ and $\mu_1=\mu_2=5\cdot 10^{-2}$, consistent with Section~\ref{sec:exp}.

\paragraph{Results and Findings}
By construction, \texttt{Lower-only} returns equilibria with comparable clearing residuals but variable policy scores: across the $10$ starts, $\mathcal{P}$ varies widely while $\|\X\Y-\M\|_F/\|\M\|_F$ is essentially flat. 
This empirically confirms the non-uniqueness that motivates the bi-level formulation. \texttt{Binno}, by contrast, attains the same residual within the \texttt{Lower-only} range while dominating every \texttt{Lower-only} trial on $\mathcal{P}$ and $\|\Y\|_0$. 
\texttt{Binno} requires more runtime, consistent with the observations in Section~\ref{subsubsec:synth_sweep}.

\begin{table}[h!]
\caption{RMC on UCI Online Retail data ($m{=}60$, $n{=}100$, $r{=}10$, $k_\star{=}3$; mean$\pm$std over 10 trials).
All methods (MaxIter = 2000) achieve the same market clearing residue of $0.5380\!\pm\!0.000$ in the end.
}
\label{tab:RMC}
\centering
\small
\begin{tabular}{r|ccc}
 & Policy score $\mathcal{P}$  & Sparsity of $\mathbf{Y}$ & Time (s) \\\hline
\texttt{Binno} & $19.5002\!\pm\!0.3066$ & $0.022\!\pm\!0.004$ & $0.205\!\pm\!0.007$ \\
\texttt{BiG-SAM} &  $21.9248\!\pm\!0.3405$ & $0.026\!\pm\!0.005$ & $0.175\!\pm\!0.007$ \\
\texttt{Lower-only} (PALM) &  $21.9522\!\pm\!0.3372$ & $0.026\!\pm\!0.004$ & $0.131\!\pm\!0.003$ \\
\end{tabular}
\end{table}

\begin{figure}[h!]
\centering
\includegraphics[width=0.55\linewidth]{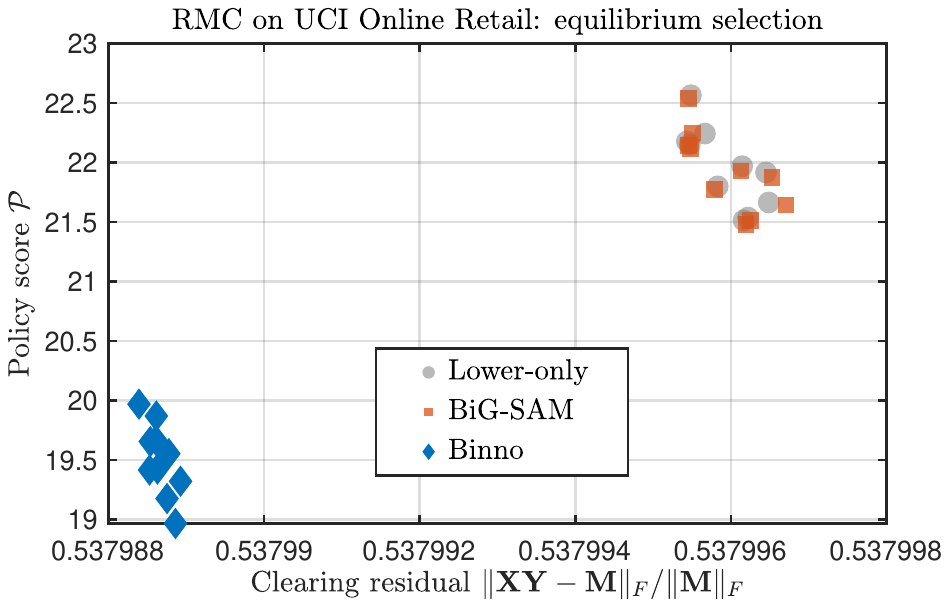}
\caption{The results of the methods at the last iteration for all the 10 trials.
Binno consistently output the lowest policy score $\cP$.}
\label{fig:econ}
\end{figure}

Note that $\mathcal{P}$ is preferably low  for mathematical and regulatory reasons:
\begin{itemize}[leftmargin=15pt,itemsep=0pt]
    \item \textbf{Mathematical Minimization:} The bi-level problem is an $\text{argmin}$ over the upper-level objective.
    So, the algorithm's primary goal is to find an equilibrium $(\mathbf{X}, \mathbf{Y})$ that yields the smallest possible value for $\mathcal{P}$ while satisfying the market-clearing conditions.
    
    \item \textbf{Low-Rank Production ($\|\mathbf{X}\|_*$):} 
    A low $\mathcal{P}$ implies that the complex market of $m$ goods can be explained by a small number of latent ``production archetypes'', making it easier for a regulator to oversee supply patterns.

    \item \textbf{Sparse Allocation ($\|\mathbf{Y}\|_1$):} A low $\mathcal{P}$ ensures that each buyer is served 
    by few sellers rather than a complex network of transactions, improving economic \textit{traceability} and auditability.
\end{itemize}
The result showed that
\texttt{Binno} is able to \textbf{consistently} select a better market equilibrium that minimizes political score, while making explainable economic outcome, among many valid clearing solutions.
}

\section{Conclusion}\label{sec:conc}
In this paper, we propose a new approach {for solving} nonconvex and nonsmooth bi-level optimization problems.
We introduce a novel algorithm, {\texttt{Binno}}, {grounded in} proximal point methods, descent conditions, and variational properties of the involved functions.
This framework allows \texttt{Binno} to preserve the descent property of the overall {solution}.

We also present practical applications {of \texttt{Binno}} to the sparse low-rank approximation problem, which frequently arises in real-world scenarios where {one extracts} meaningful information from large data matrices while maintaining a sparse representation.

Experiments on both synthetic and real datasets demonstrate the effectiveness of \texttt{Binno} compared to several state-of-the-art algorithms {in this field, showing that it outperforms traditional methods. In addition, the comparison with the bi-level baseline and PALM, highlights the advantages of the proposed coupled descent mechanism in nonconvex settings, particularly in terms of robustness and solution quality. The application to regularized market clearing further illustrates the flexibility of the framework, showing that \texttt{Binno} can effectively select economically meaningful equilibria among multiple feasible solutions. These results suggest that Binno provides a general and versatile approach for structured bi-level optimization across different application domains.}

\section*{Acknowledgment \& Funding}
F.E., L.S. are members of the Gruppo Nazionale Calcolo Scientifico - Istituto Nazionale di Alta Matematica (GNCS-INdAM). 
F.E., and L.S. are partially supported by ``INdAM - GNCS Project", CUP: E53C24001950001.\\ 
F.E. are supported by Piano Nazionale di Ripresa e Resilienza (PNRR), Missione 4 ``Istruzione e Ricerca''-Componente C2 Investimento 1.1, ``Fondo per il Programma Nazionale di Ricerca e Progetti di Rilevante Interesse Nazionale'', Progetto PRIN-2022 PNRR, P2022BLN38, Computational approaches for the integration of multi-omics data. CUP: H53D23008870001.

\bibliographystyle{elsarticle-num}
\bibliography{refs} 

\section*{Appendix: RMC in \texttt{Binno}}
RMC is \eqref{eq:Binno} with
$
f_1(\X) = \lambda_1 \|\X\|_*,\ g_1(\Y) = \lambda_2 \|\Y\|_1,\ 
f_2(\U) = \mu_1 \|\U\|_1$, $ g_2(\V) = \mu_2 \|\V\|_1,
$ with $H$ as above. 
All regularizers are convex, proper, l.s.c.\ and have bounded subgradients in the sense of Lemma~\ref{lem:nesterov}, and the assumptions of Theorem~\ref{th:descend_condition} hold.
We have that the $\X$-update is SVT on $\X_u$ and soft-thresholding on $\X_\ell$, 
followed by $\X_{k+1} = \alpha_k \X_u + (1-\alpha_k)\X_\ell$; the $\Y$-update is similar. Running \texttt{Binno} on \eqref{eq:RMC} therefore only requires swapping two prox operators in the \eqref{eq:SLRF} implementation.

The constants in thi setting are for $H$: $L_1 = \|\Y\Y^\top\|_2$ and $L_2 = \|\X^\top\X\|_2$ (Lemma~\ref{lem:H_bismooth2}). 
The subgradient bounds (Propositions~\ref{prop:c1}--\ref{prop:c2}) give
\[
c_1 = 2\lambda_1,\quad c_2 = \mu_1\sqrt{mr},\quad c_3 = \lambda_2\sqrt{rn},\quad c_4 = \mu_2\sqrt{rn}.
\]
Feeding these into Theorems~\ref{th:alpha_range}--\ref{th:beta_range}, gives  feasibility intervals for $\alpha_k,\beta_k$ and a lower bound on $\nu$. For \eqref{eq:RMC} the $\alpha_k$-feasibility condition reads: if 
\[
\nu \;\geq\; \tfrac{1}{\sqrt{(\|\Y\Y^\top\|_2+\mu_1\sqrt{mr})(\|\Y\Y^\top\|_2+2\lambda_1)}-\|\Y\Y^\top\|_2},
\]
then $\alpha_k\in\big[\tfrac{1/\nu+\|\Y\Y^\top\|_2}{2\lambda_1+1/\nu+2\|\Y\Y^\top\|_2},\,\tfrac{\mu_1\sqrt{mr}+\|\Y\Y^\top\|_2}{\mu_1\sqrt{mr}+1/\nu+2\|\Y\Y^\top\|_2}\big]$. The $\beta_k$ bound follows analogously from Theorem~\ref{th:beta_range} with $c_3,c_4$ substituted.

\end{document}